\documentclass[a4paper, 12pt]{amsart}

\usepackage{amssymb, amsthm}
\usepackage[backref]{hyperref}
\usepackage[alphabetic,backrefs,lite]{amsrefs}
\usepackage{amscd}   
\usepackage[all]{xy} 
\usepackage{graphicx}
\usepackage{mathrsfs}
\usepackage{enumerate}

\usepackage{tikz}
\usepackage{verbatim}

  \tikzstyle{block} = [rectangle, draw,
      text width=10em, text centered, minimum height=4em]
  \tikzstyle{line} = [draw, -latex']

\usepackage{amssymb, amsmath, amsthm}
\usepackage[backref]{hyperref}
\usepackage[alphabetic,backrefs,lite]{amsrefs}
\usepackage{amscd}   
\usepackage[all]{xy} 
\usepackage{graphicx}
\usepackage{mathrsfs}
\definecolor{labelkey}{rgb}{0,0,1}
\usepackage{}

\usepackage{tikz}

\newtheorem{lemma}{Lemma}[section]
\newtheorem{theorem}[lemma]{Theorem}
\newtheorem{proposition}[lemma]{Proposition}
\newtheorem{prop}[lemma]{Proposition}
\newtheorem*{prop*}{Proposition}
\newtheorem{cor}[lemma]{Corollary}
\newtheorem{conj}[lemma]{Conjecture}

\newtheorem{claim*}{Claim}

\newtheorem{defn}[lemma]{Definition}

\theoremstyle{definition}
\newtheorem{remark}[lemma]{Remark}

\makeatletter
\def\namedlabel#1#2{\begingroup
    #2%
    \def\@currentlabel{#2}%
    \phantomsection\label{#1}\endgroup
}
\makeatother

\def\O{\mathcal{O}}
\def\P{\mathbb{P}}

\newcommand{\C}{{\mathbb C}}
\newcommand{\F}{{\mathbb F}}

\newcommand{\Q}{{\mathbb Q}}

\newcommand{\Z}{{\mathbb Z}}
\newcommand{\N}{{\mathbb N}}

\newcommand{\T}{\mathbb T}

\newcommand{\p}{{\mathbb B}_{p, \infty}}

\newcommand{\Kbar}{{\overline{K}}}

\newcommand{\calA}{{\mathcal A}}

\newcommand{\calN}{{\mathcal N}}
\newcommand{\calO}{{\mathcal O}}

\newcommand{\frakf}{{\mathfrak f}}
\newcommand{\frakg}{{\mathfrak g}}

\newcommand{\frakp}{{\mathfrak p}}
\newcommand{\frakq}{{\mathfrak q}}

\newcommand{\frakD}{{\mathfrak D}}

\newcommand{\frakN}{{\mathfrak N}}

\newcommand{\frakP}{{\mathfrak P}}
\newcommand{\frakQ}{{\mathfrak Q}}

\DeclareMathOperator{\tr}{Tr}

\DeclareMathOperator{\Frob}{Frob}

\DeclareMathOperator{\End}{End}

\DeclareMathOperator{\Aut}{Aut}

\DeclareMathOperator{\Norm}{ Norm}

\DeclareMathOperator{\Cl}{Cl}

\DeclareMathOperator{\SL}{SL}
\DeclareMathOperator{\GL}{GL}

\DeclareMathOperator{\Rad}{Rad}

\usepackage{color}

\numberwithin{equation}{section}
\numberwithin{table}{section}

\title{Solving equations of signature $(p,p,2)$ with coefficients over number fields}

\author{Begum Gulsah Cakti}
\author{Erman I{\c s}{\i}k}
\author{Yasemin Kara}
\author{Ekin \"Ozman}

\address{Bogazici University\\
Department of Mathematics\\
Bebek, 34342 \\ 
Istanbul, T{\" u}rk{\.I}ye\\ \newline
Galatasaray University\\
Department of Mathematics\\
Beş{\.I}ktaş, 34349 \\
Istanbul, T{\" u}rk{\.I}ye\\ \newline
Bernoulli Institute \\
Nijenborgh 9, 9747 AG \\
Groningen, the Netherlands.}

\email{begum.cakti@std.bogazici.edu.tr, bcakti@gsu.edu.tr, b.g.cakti@rug.nl}

\address{Bilkent University\\
Department of Mathematics\\
{\c C}ankaya, 06800\\
Ankara, T{\" u}rk{\.I}ye.}

\email{erman.isik@bilkent.edu.tr}

\address{Bogazici University\\
Department of Mathematics\\
Bebek, 34342 \\
Istanbul, T{\" u}rk{\.I}ye.}

\email{yasemin.kara@bogazici.edu.tr}

\address{Bernoulli Institute \\
Nijenborgh 9, 9747 AG \\
Groningen, the Netherlands.}

\email{e.ozman@rug.nl}

\date{}


\begin{document}

\maketitle

\begin{abstract} 

Using the modular method, we study solutions to the Diophantine equation
$$Aa^p+Bb^p=Cc^2$$
over number fields. We first prove an asymptotic result for general number fields satisfying an appropriate $S$-unit condition by assuming some standard conjectures in the case of fields that are not totally real. Specifically, we verify that this condition holds for an infinite family of real quadratic fields. Outside the asymptotic setting, we also obtain  effective results. In particular, for the
equation
$$a^p+db^p=c^2$$
over $K= \Q(\sqrt{-d})$ with $d \in \{3, 11, 19, 43, \}$ and  $K= \Q(\sqrt d)$ with $d \in \{3, 5, 11, 13, 19, 29\}$, we find explicit bounds (depending on $d$)
such that no non-trivial solutions of a certain type exist whenever $p$ exceeds these bounds.

\end{abstract}

\section{Introduction}

The proof of Fermat's Last Theorem in 1995 marked a turning point in number theory, introducing the modular method—a powerful technique that connects elliptic curves and modular forms to solve Diophantine equations. Building on this foundation, many researchers have achieved significant progress in establishing asymptotic versions of Fermat's result across various contexts. For totally real fields, work by Freitas and Siksek \cite{FS} demonstrated that solutions to the standard Fermat equation vanish for sufficiently large prime exponents, with Deconinck \cite{HD} later extending this to versions of the standard Fermat equation with coefficients. More recently, these asymptotic results have been generalized to arbitrary number fields under modularity assumptions, and analogous findings have emerged for generalized Fermat equations \cite{KO}.

Despite these advances, similar progress for other Fermat-type equations has been limited. While equations such as $x^p + y^p = z^2$ have been extensively studied over the rationals by mathematicians including Darmon, Merel, Bennett, and Poonen, \cite{DM97,P98,bs04}, generalizations to higher-degree number fields remain scarce. Recent work has begun to address this gap, applying modular methods to prove asymptotic results for certain classes of solutions to equations like $x^p + y^p = z^2$ and $x^p + y^p = z^3$ over certain number fields in \cite{IKO}, \cite{isik2023ternary} and in \cite{moc22} for totally real fields. Similarly, asymptotic results for particular types of solutions to the first equation with coefficients of specific forms over totally real fields are given in \cite{KS24aa,KS24bb}. In \cite{kara2025non}, the authors have extended the asymptotic results  to the generalized equation $Aa^p + Bb^p = Cc^3$ for number fields $K$ satisfying a suitable $S$-unit condition.  
These developments follow asymptotic approaches analogous to those used for the classical and generalized Fermat equations, thereby expanding our understanding of Diophantine problems beyond rational solutions. 

In this paper, we prove that the equation $Aa^p + Bb^p = Cc^2$ does not have asymptotic solutions for number fields $K$ satisfying a suitable $S$-unit condition by combining and generalizing the approaches of  \cite{isik2023ternary} and \cite{moc22}.  In particular, we show that for certain real quadratic fields, this condition is satisfied.

   Beyond asymptotic solutions, a natural question concerns effective bounds for equations with specific coefficients. This requires explicit form elimination, which is often not possible or feasible. However, for certain families of equations, this can be achieved. In this paper, we provide effective results regarding solutions to Diophantine equations of signature $(p,p,2)$ and $(p,p,3)$ with coefficients. More precisely, we show that for the equation 
\[
a^p + d b^p = c^2
\]
over certain real and imaginary quadratic fields, there exist explicit bounds (depending on $d$) such that there are no non-trivial solutions of a particular type whenever $p$ exceeds these bounds.

 \subsection*{Our Results}

Let $K$ be a number field and let $\mathcal{O}_K$ denote its ring of integers and $\mathcal{O}_K^\times$ denote the group of units in $\mathcal{O}_K$.   For \textit{pairwise coprime, odd}  elements $A,B,C$  of $\calO_{K}$
and a prime number $p$, we refer the equation
 \begin{equation}\label{maineqn}
 Ax^p+By^p=Cz^2
 \end{equation}
 as \textit{the generalized Fermat equation over K with coefficients} $A,B,C$ \textit{and signature} $(p,p,2)$.  A solution $(a,b,c)$ is
 called \textbf{trivial} if $abc=0$, otherwise \textbf{non-trivial}. Moreover, a solution $(a,b,c)$ of Equation \ref{maineqn} is called \textbf{primitive} if $a$, $b$ and $c$ are pairwise coprime.

Note that if $a, b, c \in \calO_K$
satisfy $Aa^p+Bb^p=Cc$, where $p$ is an odd prime, then $(ac, bc, c^{\frac{p+1}{2}})$ is a nonprimitive solution to \eqref{maineqn}. Therefore, we will consider only primitive solutions to \eqref{maineqn}.  In \cite{DG}, it was shown that the number of integer solutions to Equation \ref{maineqn} is finitely many. The argument used in the proof is ineffective, meaning that it does not give an algorithm to find all possible solutions. 

Before presenting our asymptotic and explicit results for Equation \ref{maineqn}, we set up the following notation throughout the paper. 

\begin{itemize}
    \item $R:=\Rad(ABC)=\displaystyle{\prod_{\substack{\frakq\mid ABC\\ \frakq \subset \O_K \text{ a prime}}}\frakq}$,
    \item $  S_{K}:=\{\frakP: \frakP \text{ is a prime of }K\text{ above }2R\} $,
    \item $  T_{K}:=\{\frakP: \frakP \text{ is a prime of }K\text{ above }2\} $.
\end{itemize}
 Let $\mathcal{O}_{S_K}$ denote the ring of $S_K$-integers and let $\mathcal{O}_{S_K}^\times$ denote the group of $S_K$-units. Further, let $\text{Cl}_{S_K}(K)$ denote the quotient $\text{Cl}(K)/\langle[\mathfrak{P}]\rangle_{\mathfrak{P}\in S_K}$ and $\text{Cl}_{S_K}(K)[n]$ denote its $n$-torsion points.\\ 

Our asymptotic results are as follows. 

\begin{theorem}\label{maintheoremA}
   Let $K$ be a number field with $\Cl_{S_K}(K)[2]=\{1\}$. In the case where  $K$ is not totally real, we assume both Conjecture~\ref{conj1} and Conjecture~\ref{conj2} hold for $K$. Suppose that there exists some distinguished prime $\widetilde{\frakP}\in T_{K}$ such that every solution $(\alpha,\beta,\gamma)\in \left(\calO_{S_K}^\times \right)^2\times \calO_{S_K}$ to the $S$-unit equation
    \begin{equation}\label{unitequation}
        \alpha+\beta=\gamma^2
    \end{equation}    
satisfies  $|v_{\widetilde{\frakP}}(\alpha/\beta)|\leq 6v_{\widetilde{\frakP}}(2)$. Then, there is a constant $V=V(A,B,C,K)$ depending only on $K,A,B$ and $C$ such that the equation $ Ax^p+By^p=Cz^2$ does not have any non-trivial primitive solution $(a,b,c)$ in $\calO_K^3$ with $\widetilde{\frakP} \mid b$  for $p > V$. 
\end{theorem}

Theorem \ref{maintheoremA} generalizes \cite[Theorem 3]{moc22} and \cite{KS24aa,KS24bb}, which are restricted to totally real number fields and certain coefficients. The generality in our setting is achieved by employing a slightly modified level-lowering technique, adapted from \cite[\S 7]{SS} and \cite[\S 3]{isik2023ternary}, originally developed for different Fermat equations. A slight downside of this method is that we need to assume two conjectures (Conjecture~\ref{conj1} and Conjecture~\ref{conj2}) in the general setting. In the totally real case, one can avoid such conjectures using Proposition \ref{modularityoffreycurveovertotallyreal} and Theorem \ref{thm:ES}.

Furthermore, we obtain the following theorem by imposing some local constraints.

\begin{theorem}\label{maintheoremB}
   Let $K$ be a number field with $2\nmid h_K^+$ in which $2$ is inert. In the case where $K$ is not totally real, we assume both Conjecture~\ref{conj1} and Conjecture~\ref{conj2} hold for $K$. Let $\frakP$ be the only prime above $2$. Suppose that every solution $(\alpha,\gamma)\in \calO_{S_K}^\times\times \calO_{S_K}$ with $v_{\frakP}(\alpha)\geq 0$ to the $S$-unit equation
    \begin{equation}\label{unitequation2}
        \alpha+1=\gamma^2
    \end{equation}    
satisfies  $v_{\frakP}(\alpha)\leq 6$. Then, there is a constant  $V=V(A,B,C,K)$ depending only on $K,A,B$ and $C$ such that the equation $ Ax^p+By^p=Cz^2$ does not have any non-trivial primitive solution $(a,b,c)$ in  $\calO_K^3$ with $\frakP \mid b$ for $p > V$. 
\end{theorem}

\vspace{1em}
In particular, we obtain the following result when we apply the above theorem to real quadratic fields.
\begin{theorem}\label{maintheoremC}
   Let $q$ be a rational prime such that $q\geq 13$ and $q\equiv 5 \pmod{8}$, and $K=\mathbb Q(\sqrt{q})$. Let $\ell\geq 29$ be a prime satisfying $\ell \equiv 5 \pmod 8$ and $\left(\frac{q}{\ell} \right)=-1$. Then, there is a constant $B_{K,\ell}$ depending only on $K$ and $\ell$ such that the Fermat equation
\[
x^p +\ell^r y^p = z^2
\]
has no any non-trivial primitive solution $(a,b,c)$ in  $\calO_K^3$ with $\frakP=2\mathcal{O}_K\mid b$ for $p > B_{K,\ell}$.
\end{theorem}

\begin{remark}
    As we can see from the statement, the proof of Theorem~\ref{maintheoremA} will depend on the solutions to $S$-unit equation \eqref{unitequation}. It is known that, for a given number field $K$ and a finite set of primes $S$, the $S$-unit equation \eqref{unitequation} has only finitely many solutions up to an equivalence relation (see \cite[Theorem 39]{moc22}). Although there is no theoretical obstruction, it becomes practically challenging to find all the solutions when $[K:\mathbb Q]$ or $|S|$ is too large.
\end{remark}

We restrict ourselves to the case $A = C = 1$ and $B = d$ when working over imaginary and real quadratic fields for effective results.
\begin{theorem}\label{maintheoremfofrimaginaryquadraticfields}
     Let $K=\Q(\sqrt{-d})$ where $d\in \{3, 11,19, 43\}$. Assume Conjecture~\ref{conj1} holds for $K$. Write $\frakP=2\mathcal O_K$.  Then, the equation 
     $$x^p+dy^p=z^2$$
     has no non-trivial solutions $(a,b,c)$ in $\calO_K^3$ such that $(a,db,c)$ is primitive and $\frakP\mid b$ when $p>C_K$ and $p \neq d$, where \[ C_K=
\left\{\begin{array}{l}
17 \text {, if $d=3,11,43$ } \\
19\text {, if $d=19$}.
\end{array}\right.\]
 \end{theorem}
 
\begin{theorem}\label{thm1:dfermat}
    Let $K=\mathbb{Q}(\sqrt{d})$, where $d\in\{3,5,11,13,19,29\}$. Write $\frakP^e=2\mathcal{O}_K$. Then, the equation
      $$x^p+dy^p=z^2$$
    has no non-trivial solutions $(a,b,c)$ in $\calO_K^3$ such that $(a,db,c)$ is primitive and $\frakP\mid b$ when $p>C_K$ and $p \neq d$, where $$C_K=
\left\{\begin{array}{l}
17 \text {, if $d=3,5,11,13$ } \\
19\text {, if $d=19$ }\\
17\text {, if $d=29$. }
\end{array}\right.$$
\end{theorem}

\begin{remark}
    In the appendix, we will also introduce effective results for the generalized Fermat equation of signature $(p,p,3)$ with coefficients $A=C=1$, $B=d$ over certain real quadratic fields $\mathbb{Q}(\sqrt{d})$.
\end{remark}

\subsection*{Organization of the paper} 

In Section \ref{FreyCurve}, we associate an elliptic curve (Frey curve) to a putative solution of Equation \ref{maineqn} and prove relevant facts. In Section \ref{modularity}, we provide the theoretical background we need concerning modularity to prove our results. In Section \ref{level},  we prove the irreducibility of the mod $p$ Galois
representation $\overline{\rho}_{E,p}$ associated to the Frey elliptic curve and  we relate it with another representation of lower level via a new elliptic curve $E'$ whose conductor is independent of the solution. In Section \ref{asy}, we prove the asymptotic results namely Theorems \ref{maintheoremA}, \ref{maintheoremB} and \ref{maintheoremC}. In Section \ref{eff}, we give the proofs of effective results, Theorems \ref{maintheoremfofrimaginaryquadraticfields} and \ref{thm1:dfermat}. And finally, in the Appendix we prove results regarding effective solutions of another Diophantine equation $x^p+dy^p=z^3$.

All computational claims in this paper were carried out using \texttt{Magma}. The associated codes are available in the GitHub repository:
\begin{small}
    
\url{https://github.com/BGCakti/Solving-equations-of-signature-p-p-2-with-coefficients}.
\end{small}

\subsection*{Acknowledgements} 
 The first and third authors are supported by the Turkish National and Scientific Research Council (TÜBİTAK) Research Grant 122F413. The first author would like to thank the University of Groningen, Bernoulli Institute for their hospitality.

\section{Frey Curve and Related Facts}\label{FreyCurve}
In this section, we collect some facts related to the
Frey curve associated to a putative solution of Equation \ref{maineqn} and the associated Galois representation.

\subsection{Conductor of the Frey Curve}

Let $G_K$ be the absolute Galois group of a number field $K$, let $E/K$ be an elliptic curve, and  $\overline{\rho}_{E,p}$ denote the residual Galois representation attached to $E$
\begin{equation*}
    \overline{\rho}_{E,p}: G_K \longrightarrow \Aut(E[p])\cong \GL_2(\F_p).
\end{equation*}

We use $\frakq$ for an arbitrary prime of $K$, and $G_\frakq$ and $I_\frakq$ for a decomposition and inertia subgroups of $G_K$ at $\frakq$, respectively.

For a putative solution $(a,b,c)$ to the equation $Ax^p+By^p=Cz^2$ with a prime exponent $p\nmid ABC$,  we associate the Frey elliptic curve,
	 \begin{equation} \label{Frey}
	 	E:=E_{(a,b,c,A,B,C)}: Y^2=X^3+2CcX^2+BCb^pX
	 \end{equation}
	 whose arithmetic invariants are given by $$\Delta_{E}=2^{6}AB^2C^3(ab^2)^p, \; \displaystyle j_{E}=2^6\frac{(4Aa^p+Bb^p)^3}{AB^2(ab^2)^p}$$ and $$c_4(E)=16 \left(4 C^2 c^2 - 3 B C b^p\right)=2^{4}C(4Aa^p+Bb^p),\; c_6(E)=2^{6}C^2c(Bb^p-8Aa^p).$$

Let $\calN_E$ denote the conductor of $E$, and define
\begin{equation}\label{reducedconductor}
    \calN_p:=\calN_E\bigg/ {\displaystyle\prod_{\substack{\frakq\mid \mid \calN_E\\p\mid v_\frakq(\Delta_\frakq)}}}\frakq,
\end{equation}
where $\Delta_\frakq$ denotes the discriminant of a local minimal model of $E$ at $\frakq$. The following lemma collects the facts related to the Frey curve $E$.

\begin{lemma}\label{conductoroffreycurve} 
Let E be the Frey curve defined in \eqref{Frey}. Then:
 \begin{enumerate}
	\item The Frey curve $E$ is semistable away from the primes dividing $2C$ and has a $K$-rational point of order $2$. Moreover, the conductor $\mathcal{N}_E$ attached to the Frey curve $E$ is given by
    \begin{equation*}
        \mathcal{N}_E=\prod_{\substack{\frakP\mid 2C}}\frakP^{r_\frakP}\cdot \prod_{\substack{\substack{\frakq\mid ABab\\ \frakq\nmid 2}}}\mathfrak{q} \text{ and }  \calN_p=\prod_{\substack{\frakP\in S_K}}\frakP^{r'_\frakP},
    \end{equation*}
    where $0\leq r'_\frakP\leq r_\frakP\leq 2+6v_\frakP(2)$ for the primes $\frakP\in S_K$. In particular, the ideal $\calN_p$ can only be divisible by the primes of $K$ dividing $2ABC$.

   \item The Serre conductor $\frakN_E$, which is the prime-to-$p$ part of the Artin conductor of $\overline{\rho}_{E,p}$, is supported on the set $S_{K}$ and belongs to a finite set depending only on the field $K$. 
   
   \item   For large enough $p$ (depending only on $A,B,C$ and $K$), the Galois representation $\overline{\rho}_{E,p}$ is finite flat at every prime $\frakp$ of $K$ that lies above $p$. The determinant of the representation $\overline{\rho}_{E,p}$ is the mod $p$ cyclotomic character $\chi_p$.
    \end{enumerate}
\end{lemma}
\begin{proof}
  Recall that the invariants $c_4(E)$ and $\Delta_E$ are given by
    \begin{equation*}
         c_4(E)=2^{4}C(4Aa^p+Bb^p)\;\text{and}\; \Delta_E= 2^{6}AB^2C^3(ab^2)^p.
    \end{equation*} 
\begin{itemize}
    \item  If $\frakq\not\in S_{K}$ does not divide $ab$, then $v_\frakq(\Delta_E)=0$. Therefore, the model is minimal and $E$ has good reduction at $\frakq$.

    \item If  $\frakq$ is an odd prime and $\frakq \mid Aa$ or $Bb$, then the given model is minimal and $E$ has multiplicative reduction at $\frakq$. Indeed, without loss of generality, assume that  $\frakq$ is an odd prime and $\frakq \mid Aa$. We then see that $\frakq$ does not divide $CBb$. It then follows that  $c_4(E)=2^{4}C(4Aa^p+Bb^p)$ is not divisible by $\frakq$, i.e. $v_\frakq(c_4(E))=0$. Thus, the given model is minimal and $E$ has multiplicative reduction at $\frakq$.  For each prime $\frakP$ dividing $2C$, the bounds on the exponents $r_\frakP$ and $r'_\frakP$ follow from \cite[Theorem IV.10.4]{Sil94}. It immediately follows from Equation \ref{Frey} that the Frey curve $E$ has a $2$-torsion point over $K$. This finishes the proof of (1).

\end{itemize}
    
  Let $\frakp$ be a prime of $K$ lying above $p$. We have
    \begin{equation*}
        v_\frakp(\Delta_E)=v_\frakp(A)+2v_\frakp(B)+3v_\frakp(C) +pv_\frakp(a)+2pv_\frakp(b).   
    \end{equation*}

   As we assume $p\nmid ABC$, we see that $p\mid v_\frakp(\Delta_E)$. It then follows from \cite{ser87} that $\overline{\rho}_{E,p}$ is finite-flat at $\frakp$.  We can also deduce that $\overline{\rho}_{E,p}$ is unramified at $\frakq$ for odd primes $\frakq$ of $K$ such that $\frakq\nmid p$  and hence $\frakq\nmid \frakN_E$. Therefore, for all primes $\frakq\not\in S_K$, we have $\frakq\nmid  \frakN_E$. In other words, the Serre conductor $\frakN_E$ is supported on the primes in $S_K$. As the Serre conductor $\frakN_E$ divides $\calN_E$ and is divisible only by the primes $S_K$, there can be only finitely many Serre conductors and they only depend on $K$ and $A,B,C$. The statement concerning the determinant is a well-known consequence of the Weil pairing attached to elliptic curves. Hence, this completes the proof.
    \end{proof}

\subsection{Image of Inertia}

Let $K$ be a number field, and let $\frakq$ be a prime of $K$. In this section, we gather information about the image of inertia $\overline{\rho}_{E,p}(I_{\widetilde\frakP})$ for certain primes $\widetilde\frakP\in T_{K}$. This is a crucial step in controlling the behaviour of the newform obtained by level lowering at $\widetilde\frakP\in T_{K}$.

We use the following lemma for detecting when an elliptic curve has potentially multiplicative reduction.
\begin{lemma}[\cite{FS}, Lemma 3.4]\label{imageofinertia}
	Let $E$ be an elliptic curve over $K$ with $j$-invariant $j_E$.  Let $p \geq 5$ and $\mathfrak{q}\nmid p$ be a prime  
	of $K$. Then $p\mid \#\overline{\rho}_{E,p}(I_{\mathfrak{q}})$ if and only if $E$ has potentially multiplicative
	reduction at $\mathfrak{q}$ (i.e. $v_{\mathfrak{q}}(j_E)<0$) and $p \nmid v_{\mathfrak{q}}(j_E)$.
\end{lemma}   
By using the previous result we obtain:

\begin{lemma}\label{potmultred}
	Fix $\widetilde\frakP\in T_{K}$ and let $(a,b,c)\in \O_K^3$ be a non-trivial primitive solution to \eqref{maineqn} with $\widetilde\frakP\mid b$ and exponent $p>6v_{\widetilde\frakP}(2)$. Let $E$ be the Frey curve as in \eqref{Frey} associated to $(a,b,c)$, and write $j_E$ for its $j$-invariant. Then $E$ has potentially multiplicative reduction at $\widetilde\frakP$ and $p\mid \#\overline{\rho}_{E,p}(I_{\widetilde\frakP})$, where $I_{\widetilde\frakP}$ denotes an inertia subgroup of $G_{K}$ at $\widetilde\frakP$.
\end{lemma}  
\begin{proof}
    Let $\widetilde\frakP\in T_{K}$ with $\widetilde\frakP\mid b$ and exponent $p>6v_{\widetilde\frakP}(2)$. Then we have 
    \begin{align*}
         v_{\widetilde\frakP}(j_E)&=6v_{\widetilde\frakP}(2)+3v_{\widetilde\frakP}(4Aa^p+Bb^p)-pv_{\widetilde\frakP}(ab^2)\\
         &=12v_{\widetilde\frakP}(2)-2pv_{\widetilde\frakP}(b)\\
         &=2(6v_{\widetilde\frakP}(2)-pv_{\widetilde\frakP}(b)).
    \end{align*}
    Since $p>6v_{\widetilde\frakP}(2)$, we have $E$ has potentially multiplicative reduction at $\widetilde\frakP$ (i.e. $v_{\widetilde\frakP}(j_E)<0$), and $p\nmid v_{\widetilde\frakP}(j_E)$. It then follows from Lemma~\ref{imageofinertia} that $p\mid \#\overline{\rho}_{E,p}(I_{\widetilde\frakP})$.
\end{proof}

\section{Modularity of the Frey Curve}\label{modularity}
In this section, we provide the theoretical background we need concerning modularity to prove our results.

If $E$ is an elliptic curve over a number field $K$, we say that $E$ is modular if there is an isomorphism of a compatible system of Galois representations
\begin{equation*}
    \rho_{E,p}\sim \rho_{\frakf,\varpi},
\end{equation*}
where $\frakf$ is an automorphic form over $K$ of weight two with coefficient field $\Q_\frakf$ and $\varpi$ is a prime in $\Q_\frakf$ lying above $p$.

\subsection{Totally real case} In the totally real case, $\frakf$ comes from a Hilbert eigenform of parallel weight two over $K$. The following result is due to Freitas, Le Hung and Siksek \cite[Theorems 1 and 5]{siksek}:
\begin{theorem}\label{modularityovertotallyreal}
    Let $K$ be a totally real number field. There are at most finitely many $\overline{K}$-isomorphism classes of non-modular elliptic curves $E$ over $K$. Moreover, if $K$ is real quadratic, then all elliptic curves over $K$ are modular.
\end{theorem}
\begin{remark}
All elliptic curves defined over totally real number fields are
conjectured to be modular. It was proven to hold for
cubic fields in \cite{dns20} and for quartic fields not containing a square
root of $5$ in \cite{box22}.
\end{remark}

\begin{prop}\label{modularityoffreycurveovertotallyreal}
   Let $K$ be a totally real number field, and fix $\widetilde\frakP\in T_{K}$. There is some constant $W(A,B,C,K)$, depending only on $K$ and $A,B,C$, such that for any non-trivial primitive solution $(a,b,c)$ to \eqref{maineqn} with $\widetilde\frakP\mid b$ and the exponent $p> W(A,B,C,K)$, the Frey curve $E$ is modular.
\end{prop}
\begin{proof}
   The proof follows very closely the proof of \cite[Lemma 26]{moc22}. It follows from Theorem~\ref{modularityovertotallyreal} that there are at most finitely many possible $\overline{K}$-isomorphism classes of elliptic curves over $K$ which are non-modular. Let $j_1,\dots, j_n \in K$ denote the $j$-invariants of these classes. Define $\lambda:=\frac{Aa^p}{Bb^p}$ so that the $j$-invariant of $E$ is
    \begin{equation*}
        j(\lambda)=2^6(4\lambda+1)^3\lambda^{-1}.
    \end{equation*}

    We can assume that $\lambda\not\in \mathbb Q$, since these $\lambda$ would lead to $j(\lambda)\in \mathbb Q$, and it is known that all rational elliptic curves are modular. Each equation $j(\lambda)=j_i$ has at most three solutions in $K$. Thus there are values $\lambda_1,\dots, \lambda_m\in K$ (where $m\leq 3n$) such that if $\lambda\neq \lambda_k$ for all $k$, then the elliptic curve $E$ with the $j$-invariant $j(\lambda)$ is modular.
    
    If $\lambda=\lambda_k$ for some $k$, then $\tfrac{A}{B}\left(\tfrac{a}{b}\right)^p=\lambda_k$. Since $K$ is totally real and $\lambda_k \notin \{0,\pm1\}$, this equality holds if and only if $\frac{B}{A}\lambda_k \in (K^\times)^p$. For each fixed $k$, the element $\frac{B}{A}\lambda_k \in K^\times$ can be a $p$-th power for only finitely many primes $p$.
Thus, there exists a constant $V_k$ such that for all primes $p > V_k$, the above equality cannot hold.

Let $W(A,B,C,K) := \max_{k}\{V_k\}$. Then for any prime $p > W(A,B,C,K)$, we have $\lambda \neq \lambda_k$ for all $k$. Therefore, the Frey curve $E$ is modular.
\end{proof}

\begin{remark}
    The finiteness part of Theorem~\ref{modularityovertotallyreal} relies on Falting's Theorem \cite{falt83}. The bound in Falting's Theorem is ineffective, and so is the bound $W(A,B,C,K)$. Note that if $K$ is quadratic or cubic we get $W(A,B,C,K)=0$.
\end{remark}

\subsection{General case} One needs generalized modularity theorems in order to run the modular approach to solve Diophantine equations. Because of the lack of their existence in the general case, we must assume them as conjectures. In this section, we give Conjecture \ref{conj1} which is a special case of Serre's modularity conjecture. First, we fix some notations.

Let $K$ be an arbitrary number field (admitting possibly a complex embedding) with the ring of integers $\O_K$ and signature $(r,s)$.  Let $\frakN$ be an ideal of $\O_K$ and consider the locally symmetric space $Y_0(\frakN)$, defined as in \cite[\S 4]{SS}. 

\vspace{1.5mm}

For $i \in \{0,\dots , 2r + 3s\}$, consider the $i$-th cohomology group $H^i(Y_0(\frakN), \C)$. For every prime $\frakq$ coprime to the level $\frakN$, one can construct a Hecke operator $T_\frakq$ of $H^i(Y_0(\frakN), \C)$ and these operators commute with each other. We let $\T^{(i)}_\C (\frakN)$ denote the commutative $\Z$-algebra generated by these Hecke
operators inside the endomorphism algebra of $H^i(Y_0(\frakN), \C)$.

\begin{defn}
     A weight two complex eigenform $\frakf$ over $K$ of degree $i$ and level $\frakN$ is a ring homomorphism $\frakf :\T^{(i)}_\C (\frakN)\longrightarrow \C$.
\end{defn}

Note that the values of $\frakf$ are algebraic integers and they generate a number field which we shall denote by $\Q_\frakf$. We call a complex eigenform \textit{trivial} if we have $\frakf(T_\frakq) = \N\frakq + 1$ for all primes $\frakq$ coprime to the level. We call two complex eigenforms $\frakf, \frakg$ with possibly different degrees and levels \textit{equivalent} if $\frakf(T_\frakq) = \frakg(T_\frakq)$ for almost all prime ideals $\frakq$. A complex eigenform of level $\frakN$ is called \textit{new} if it is not equivalent to one whose level is a proper divisor of $\frakN$.

\vspace{1.5mm}

Let $p$ be a rational prime unramified in $K$ and coprime to $\frakN$. The cohomology group $H^i(Y_0(\frakN), \overline{\F}_p)$ also comes equipped with Hecke operators, still denoted $T_\frakq$ (we only consider the primes $\frakq$ coprime to $p\frakN$). Denote the corresponding Hecke algebra by $\T^{(i)}_{\overline{\F}_p} (\frakN)$.

\begin{defn}
      A weight two mod $p$ eigenform $\theta$ over $K$ of degree $i$ and level $\frakN$ is a ring homomorphism $\theta :\T^{(i)}_{\overline{\F}_p} (\frakN)\longrightarrow \overline{\F}_p$.
\end{defn}

We say that a mod $p$ eigenform $\theta$ of level $\frakN$ lifts to a complex eigenform if there is a complex eigenform $\frakf$ of the same degree and level and a prime ideal $\frakp$ of $\Q_\frakf$ over $p$ such that for every prime $\frakq$ of $K$ coprime to $p\frakN$ we have $\theta(T_\frakq) \equiv \frakf(T_\frakq) \pmod{\frakp}$.

\vspace{1.5mm}

  For a detailed discussion concerning complex and mod $p$ eigenforms over $K$, we refer the reader to  \cite[Sections 2 and 3]{SS}. The following conjecture is a special case of Serre's modularity conjecture over number fields.

\begin{conj}[\cite{SS}, Conjecture 3.1]\label{conj1}  Let $\overline{\rho}:G_K\rightarrow \GL_2(\overline{\mathbb{F}}_p)$ be an odd, irreducible, continuous representation with Serre conductor $\frakN$ (prime-to-$p$ part of its Artin conductor) 
	and such that $\det(\overline{\rho})=\chi_p$ is the mod $p$ cyclotomic character.   
	Assume that $p$ is unramified in $K$ and that $\overline{\rho}\mid _{G_{\Q_\frakp}}$ arises from a finite-flat group scheme over $\O_{\Q_{\frakp}}$ for every prime $\frakp\mid p$.  Then there is a weight two, mod $p$ eigenform $\theta$ over $K$ of level $\frakN$ such that for all primes $\frakq$ coprime to $p\frakN$, we have
	\[
	{\rm Tr}(\overline{\rho}({\rm Frob}_{\frakq}))=\theta(T_{\frakq}),
	\]
	where $T_{\frakq}$ denotes the Hecke operator at $\frakq$.
\end{conj}

\vspace{1.5mm}

\begin{remark}
   Given a number field $K$, we obtain a \emph{complex conjugation} for every real embedding 
$\sigma: K \hookrightarrow \mathbb R$ and every extension $\widetilde{\sigma}: \overline{K} \hookrightarrow \mathbb C$ of 
$\sigma$ as $\widetilde{\sigma}^{-1}\iota \widetilde{\sigma} \in G_K$ where $\iota$ is the usual complex conjugation. Recall that a representation $\overline{\rho}: G_K \rightarrow \GL_2(\overline{\mathbb F}_p)$ is \emph{odd} if the determinant of every complex  conjugation is $-1$.  If the number field $K$ has no real embeddings, then we immediately say that $\overline{\rho}$ is odd. It follows from Lemma \ref{conductoroffreycurve} that $\overline{\rho}_{E,p}$ is odd.
 
\end{remark}

\section{Properties of Galois Representations and Level Lowering} \label{level}

In this section, we first prove the irreducibility of the mod $p$ Galois
representation, $\overline{\rho}_{E,p}$, associated to the Frey elliptic curve and then we relate it with another representation of lower level via a new elliptic curve $E'$ whose conductor is independent of the solution.

\subsection{Irreducibility of the associated Galois representation $\overline{\rho}_{E,p}$}

We need to prove the irreducibility of the mod $p$ Galois
representation associated to the Frey elliptic curve. This is needed in the level lowering step, which will be proved in Section \ref{ES}.
The following well-known result about subgroups of $\GL_2(\mathbb F_p)$ will be used to prove the irreducibility of the Galois representation $\overline{\rho}_{E,p}$.

\begin{theorem}\label{subgroups} Let $E$ be an  elliptic curve over a number field $K$ of degree $d$ and let $G \leq \GL_2(\mathbb F_p)$ be the
	image of the mod $p$ Galois representation $\overline{\rho}_{E,p}$ of $E$.
	Then the following holds:
	\begin{itemize}
		\item If $p \mid  \#G$ then either $\overline{\rho}_{E,p}$ is reducible or $G$ contains ${\rm SL}_2(\mathbb F_p)$. In the latter case, we deduce that $\overline{\rho}_{E,p}$ is absolutely irreducible. 
		\item If $p \nmid \#G$ and $p > 15 d +1$ then $G$ is contained in a Cartan subgroup or $G$ is contained in the normalizer of a Cartan subgroup but not the Cartan subgroup itself.
		
	\end{itemize}
\end{theorem}
\begin{proof}
	For the proof, the main reference is \cite[Lemma 2]{SD}. The version above including the proof of the second part is from \cite[Propositions 2.3  and 2.6]{localglobal}.
\end{proof}

The following is Proposition 6.1 from \c{S}eng\"{u}n and Siksek \cite{SS}. We include its statement for the convenience of the reader but we will omit its proof and refer to \cite{SS} instead.

\vspace{1.5mm}

\begin{prop}\label{irredpp3}
	Let $L$ be a Galois number field and let $\frakq$ be a prime of $L$. There is a constant $B_{L,\frakq}$ such that the following is true. Let $p > B_{L, \frakq}$ be  a rational prime. Let $E/L$ be an elliptic curve that is semistable at all $\frakp \mid  p$ and has potentially multiplicative reduction at $\frakq$. Then $\overline{\rho}_{E,p}$ is irreducible. 	
\end{prop}

We apply this  to our Frey curve given in \eqref{Frey} to obtain the surjectivity of  $\overline{\rho}_{E,p}$.
\begin{cor}\label{galrepsurjective}
	Let $K$ be a number field, and fix $\widetilde\frakP\in T_{K}$.  There is a constant $D=D(A,B,C,K)$ such that if $p>D$ and 
	$(a,b,c)$ is a non-trivial primitive solution to \eqref{maineqn} with $\widetilde\frakP\mid b$, then  the Galois representation $\overline{\rho}_{E,p}$ is surjective.
\end{cor}
\begin{proof} It follows by Lemma \ref{potmultred} that if $p>6v_{\widetilde\frakP}(2)$, the Frey curve $E$ has potentially multiplicative reduction at $\widetilde\frakP$.  Also, $E$ is semistable away from the primes dividing $2C$ by Lemma \ref{conductoroffreycurve}.  Let $L$ be the Galois closure of $K$, and let $\widetilde\frakq$ be a prime	of $L$ above $\widetilde\frakP$. Now, by applying Proposition \ref{irredpp3}, we get a constant $B_{L,\widetilde\frakq}$ such that $\overline{\rho}_{E,p}$ is irreducible whenever $p>B_{L,\widetilde\frakq}$. Note that there are only finitely many choices of $\widetilde\frakq$ in $L$ dividing $\widetilde\frakP$, and $L$ only depends on $K$.  Hence, we can obtain a constant depending only on $A,B,C, K$ and we denote it by $D=D(A,B,C,K)$. 	We also enlarge $D$, if necessary, so that $D>6v_{\widetilde\frakP}(2)$ (since we applied Lemma~\ref{potmultred}) and deduce that $\overline{\rho}_{E,p}$ is irreducible when $p>D$ by Proposition \ref{irredpp3}.

    \vspace{1.5mm}
    
    Now, we apply Lemma \ref{potmultred} again and see that the image of $\overline{\rho}_{E,p}$ contains an element of order $p$.  By Theorem \ref{subgroups} any subgroup of $\GL_2(\F_p)$ 
	having an element of order $p$ is either reducible or contains ${\rm SL}_2(\F_p)$. As $p>D>6v_{\frakP}(2)$, which implies that $\overline{\rho}_{E,p}$ is irreducible, the image	contains ${\rm SL}_2(\F_p)$.  Finally, we can ensure that $K\cap\Q(\zeta_p)=\Q$ by taking $D$ large enough if needed.	Hence, $\chi_p=\det(\overline{\rho}_{E,p})$ is surjective giving the following short exact sequence 
 \begin{equation*}
     1 \longrightarrow \SL_2(\F_p) \longrightarrow \overline{\rho}_{E,p}(G_K) \xrightarrow{\;\;\det\;\;} \F_p^\times \longrightarrow 1,
 \end{equation*}
 which completes the proof.
\end{proof}

\subsection{Level Lowering and Eichler-Shimura}
\label{ES}

For our asymptotic results, namely for Theorems \ref{maintheoremA}, \ref{maintheoremB} and \ref{maintheoremC}, we will need to relate the Galois representation $\overline{\rho}_{E,p}$ attached to the Frey curve $E$ with another representation of lower level via a new elliptic curve $E'$ whose conductor is independent of the solution.  We will use $E'$ to arrive at a contradiction in our asymptotic result.  The rest of this section is devoted to proving Theorem \ref{levellowerinandeichlershimura} which gives the properties of $E'$. The key ideas of the aforementioned theorem depend on the so-called Level-Lowering Theorem (Theorem \ref{level lowering}) and Eichler-Shimura Theorem (Theorem \ref{thm:ES}) when the field $K$ is totally real.

\subsubsection{Totally real case}\label{eichlershimuratotallyrealcase} We present a level-lowering result, which can be viewed as a generalization of Ribet's Level Lowering Theorem, by Freitas and Siksek \cite[Theorem 7]{FSANT} derived from the work of Fujira, Jarvis, and Rajaei.
For an elliptic curve $E/K$ of conductor $\calN_E$, we recall the definition of $\calN_p$ :

\begin{equation}
    \calN_p:=\calN_E\bigg/ {\displaystyle\prod_{\substack{\frakq\mid \mid \calN_E\\p\mid v_\frakq(\Delta_\frakq)}}}\frakq,
\end{equation}
where $p$ be a rational prime and $\Delta_\frakq$ is the discriminant of a local minimal model for $E$ at $\frakq$.
\begin{theorem} (Level-Lowering)\label{level lowering}
 	Let $K$ be a totally real number field and $E/K$ an elliptic curve of conductor $\calN_E$. Let $p$ be a rational prime. For a prime $\frakq$ of $K$, let $\Delta_\frakq$ denote the minimal discriminant of $E$ at $\frakq$, and $\calN_p$ be as in \eqref{reducedconductor}.  Suppose that the following statements hold:
 	\begin{enumerate}[(i)]
 		\item $p\geq 5$, the ramification index $e(\mathfrak{q}/p)<p-1$ for all $\mathfrak{q}\mid p$, and $\mathbb{Q}(\zeta_p)^+ \nsubseteq K$,
 		\item $E$ is modular,
 		\item $\overline{\rho}_{E,p}$ is irreducible,
 		\item $E$ is semistable at all $\mathfrak{q}\mid p$,
 		\item $p\mid v_{\mathfrak{q}}(\Delta_{\mathfrak{q}})$ for all $\mathfrak{q}\mid p$.
 	\end{enumerate}
 	Then there is a Hilbert eigenform $\mathfrak{f}$ of parallel weight two that is new at level $\mathcal{N}_p$ and some prime $\varpi$ of $\Q_{\mathfrak{f}}$ such that $\varpi\mid p$ and 
 	$\overline{\rho}_{E,p} \sim \overline{\rho}_{\mathfrak{f},\varpi}$.
 \end{theorem}

  Eichler-Shimura conjecture, which can be viewed as a converse of the modularity, states that for a Hilbert newform $\frakf$ of level $\calN$ and parallel weight two with rational eigenvalues, there is an elliptic curve $E_\frakf /K$ with conductor $\calN$  having the same $L$-function as $\frakf$. Below, we try to summarize some of the cases in which the conjecture holds true. 

\begin{theorem} \label{thm:ES}
	Let $K$ be a totally real field and let $\frakf$  be a Hilbert
	newform over $K$ of level $\mathcal N$ and parallel weight $2$, such that $\Q_\frakf = \Q$. 
	
	\begin{enumerate}
		\item If the degree of $K$ over $\Q$ is odd then there is an elliptic curve $E_\frakf/K$ of conductor $\mathcal N$ with the same $L$-function as $\frakf$. (Blasius \cite{Blas04}, Hida \cite{Hida})
		
		\item Let $p$ be an odd prime and $\frakq\nmid p$ be a prime of $K$.  Suppose $\bar{\rho}_{E,p}$ is irreducible, and $\bar{\rho}_{E,p} \sim \bar{\rho}_{\frakf,p}$ with the following properties:
		\begin{enumerate}
			\item $E$ has potentially multiplicative reduction at $\frakq$,
			\item $p \mid  \# \bar{\rho}_{E,p}(I_\frakq) $ where $I_\frakq$ denotes the inertia group at $p$,
			\item $p \nmid (\text{Norm}_{K/\Q}(\frakq) \pm 1)$,
		\end{enumerate}
		then there is an elliptic curve $E_\frakf/K$ of conductor $\mathcal N$ with the same $L$-function as $\frakf$. (Freitas-Siksek Corollary 2.2 in \cite{FS})
	\end{enumerate}
\end{theorem}

\subsubsection{General case} 

Now, we apply Conjecture~\ref{conj1} to relate the residual representation $\overline{\rho}_{E,p}$ to a weight two, mod $p$ eigenform $\theta$ over $K$ of level $\frakN_E$. One of the main steps toward the proof is to lift mod $p$ eigenforms to complex ones. The following result is due to \c{S}eng\"{u}n and Siksek.

\begin{prop}[\cite{SS}, Proposition 2.1]\label{eigenform}
	Let $\frakN$ be an ideal of $\calO_K$ .There is an integer $B(\frakN)$ depending only on $\frakN$ such that for any prime $p>B(\frakN)$, every weight two, mod $p$ eigenform of level $\frakN$ lifts to a complex one.
\end{prop}

In addition to Conjecture~\ref{conj1}, we will use a special case of a fundamental conjecture from the Langlands Programme since not everything is known about the connection between eigenforms and elliptic curves over general number fields. Recall that a simple abelian surface $A$ over $K$ whose algebra $\End_K(A)\otimes_{\Z}\Q$ of $K$-endomorphisms is an indefinite division quaternion algebra $D$ over $\Q$ is called a \textit{fake elliptic curve} (see \cite[\S 4]{SS}).

\begin{conj}[\cite{SS}, Conjecture 4.1]\label{conj2}
	Let $\frakf$ be a weight two complex eigenform over $K$ of level $\frakN$ that is non-trivial and new.  If $K$ has some
	real place, then there exists an elliptic curve $E_{\frakf}/K$ of conductor $\frakN$ such that 
	\begin{equation}\label{c2eqn}
		\#E_{\frakf}(\O_K/\frakq)=1+\Norm(\frakq)-\frakf(T_{\frakq})\quad\mbox{for all}\quad\frakq\;\nmid\;\frakN.
	\end{equation}
	If $K$ is totally complex, then there exists either an elliptic curve $E_{\frakf}$ of conductor $\frakN$ satisfying (\ref{c2eqn})
	or a fake elliptic curve 
	$A_{\frakf}/K$, of conductor $\frakN^2$, such that
	\begin{equation}
		\#A_{\frakf}(\O_K/\frakq)=(1+\Norm(\frakq)-\frakf(T_{\frakq}))^2\quad\mbox{for all}\quad\frakq\;\nmid\;\frakN.
	\end{equation}
\end{conj}	

\vspace{1.5mm}

Let $K$ be a number field which is not totally real and assume Conjecture~\ref{conj1} and Conjecture~\ref{conj2} hold for $K$. Under these assumptions, we prove the following lemma which is an important step in the proof of Theorem \ref{levellowerinandeichlershimura}.

\begin{lemma}\label{eichlershimura}
	There is a non-trivial, new (weight two) complex eigenform $\frakf$ which has an associated elliptic curve  $E_{\frakf}/K$ of conductor $\frakN'$ dividing $\frakN_E$.
\end{lemma}
\begin{proof}
     We first show the existence of such an eigenform $\frakf$ of level $\frakN_E$. By Corollary \ref{galrepsurjective}, we see that the representation $\overline{\rho}_{E,p}:G_K\rightarrow \GL_{2}(\mathbb{F}_p) $ is surjective, hence is absolutely irreducible for $p>D=D(A, B, C, K)$.  Now, we apply  Conjecture \ref{conj1} to deduce that there is 
	a weight two mod $p$ eigenform $\theta$ over $K$ of level $\frakN_E$, with $\frakN_E$ as in Lemma \ref{conductoroffreycurve}, such that for all primes $\frakq$ coprime to $p\frakN_E$, we have 	\[\tr(\overline{\rho}_{E,p}({\rm Frob}_{\frakq}))=\theta(T_{\frakq}).\]

We also know from the same lemma that there are only finitely many possible levels $\frakN_E$.  Thus, by taking $p$
	large enough, see Proposition \ref{eigenform}, for any level $\frakN_E$
	there is a weight two complex eigenform $\frakf$ of level $\frakN_E$ which is a lift of $\theta$. Note that since there are only finitely many such eigenforms $\frakf$
	and they depend only on $K,A,B,C$, from now on we can suppose that every constant depending on these eigenforms depends
	only on $K,A,B,C$.

 \vspace{1.5mm}

 Next, following exactly the same steps as in \cite[Lemma 7.2]{SS} we can see that if $\Q_\frakf\neq\Q$ then  there is a constant $C_{\frakf}$ depending only on $\frakf$ such that 
	$p<C_{\frakf}$.  Therefore, by taking $p$ sufficiently large, we assume that $\Q_\frakf=\Q$. 
	In order to apply Conjecture \ref{conj2}, we need to
	show that $\frakf$ is non-trivial and new.  As $\overline{\rho}_{E,p}$ is irreducible, the eigenform $\frakf$ is non-trivial.  
	If $\frakf$ is new, we are done.  If not, we can replace it with an equivalent new eigenform of a smaller level.  Therefore,
	we can take $\frakf$ new with level $\frakN'$ dividing $\frakN_E$.
	Finally, we apply Conjecture \ref{conj2} and obtain that $\frakf$ either has an associated elliptic curve $E_{\frakf}/K$ of conductor $\frakN'$, or has an associated fake elliptic curve $A_{\frakf}/K$ of conductor $\frakN'^2$.

 \vspace{1.5mm}
	
	By Lemma \ref{fakeellipticcurve} below, if $p>24$, then $\frakf$ has an associated elliptic curve $E_{\frakf}$. As a result, we can assume that $\overline{\rho}_{E,p}\sim \overline{\rho}_{E',p}$ where $E'=E_{\frakf}$ is an elliptic curve with conductor $\frakN'$ dividing $\frakN_E$.  

\end{proof}

\begin{lemma}\cite[Lemma 7.3]{SS}\label{fakeellipticcurve}
	If $p>24$, then $\frakf$ has an associated elliptic curve $E_{\frakf}$.
\end{lemma}

The following lemma will allow us to take $p$ large enough so that $E'$ has a $K$-rational point
of order $2$.
  \begin{lemma}\label{full2torsion}
  If $E'$ does not have a nontrivial $K$-rational point of order $2$, and is not isogenous to an elliptic curve with a nontrivial $K$-rational point of order $2$, then $p<C_{E'}$ where $C_{E'}$ is a constant depending only on $E'$. 
		\end{lemma}
  \begin{proof}
         By \cite[Theorem 2]{Katz81} there are infinitely many primes $\frakq$ such that $\# E'(\F_\frakq)\not \equiv 0 \pmod 2$.  Fix such a prime $\frakq\not\in S_K$ and note that $E$ is semistable at $\frakq$.  If $E$ has good reduction at $\frakq$, then $\# E(\F_\frakq)\equiv \# E'(\F_\frakq) \pmod p$. Since $\# E(\F_\frakq)$ is divisible by $2$, the difference $\# E(\F_\frakq)- \# E'(\F_\frakq)$, which is divisible by $p$, is nonzero.  As the difference belongs to a finite set depending on $\frakq$, the rational prime $p$ becomes bounded.  If $E$ has multiplicative reduction at $\frakq$, we obtain 
      \[
		\pm(\Norm(\frakq)+1)\equiv a_{\frakq}(E') \pmod p
		\]
		by comparing the traces of Frobenius.  We see that this difference being also nonzero and depending only on $\frakq$ gives a bound for $p$.
  \end{proof}

Now, we state and prove the main theorem of this section.

   \begin{theorem}\label{levellowerinandeichlershimura}
Let $K$ be a number field, and let $\widetilde\frakP\in T_{K}$ be a fixed prime. If $K$ is not totally real, we assume both Conjecture~\ref{conj1} and Conjecture~\ref{conj2} hold for $K$. There is a constant $V=V(A, B, C, K)$ depending only on $K$ and $A,B,C$ such that, for $p>V$, the following holds: 
Let $(a,b,c)$ be a putative solution to \eqref{maineqn} with $\widetilde\frakP\mid b$. Write $E$ for the Frey curve in \eqref{Frey}.
Then there is an elliptic curve $E'$ over $K$ such that:
\begin{enumerate}[(i)]
\item the elliptic curve $E'$ has good reduction away from $S_K$;
\item $\overline{\rho}_{E,p} \sim \overline{\rho}_{E',p}$;
\item $E'$ has a $K$-rational point of order $2$;
\item for $\widetilde\frakP\in T_{K}$, $v_{\widetilde\frakP}(j_{E'})<0$ where $j_{E'}$ is the $j$-invariant of $E'$.
\end{enumerate}
\end{theorem}  
\begin{proof}

First, we will consider the totally real case. Let $K$ be a totally real number field. It follows from  Lemma \ref{conductoroffreycurve} that $E$ is semistable outside the primes dividing $2C$. Moreover, $\overline{\rho}_{E,p}$ is irreducible by Corollary~\ref{galrepsurjective}, and $E$ is modular by Proposition~\ref{modularityoffreycurveovertotallyreal} after taking $V$  sufficiently large. We then apply Theorem \ref{level lowering} (after enlarging $V$ to ensure that $\mathbb{Q}(\zeta_p)^+ \nsubseteq K$) and Lemma \ref{conductoroffreycurve} to obtain  $\overline{\rho}_{E,p}\sim \overline{\rho}_{\mathfrak{f},\varpi}$ for some Hilbert newform  $\frakf$ of level $\mathcal{N}_p$ and some prime $\varpi\mid p$ of $\Q_{\mathfrak{f}}$ where $\Q_\frakf$ denotes the field generated by the Hecke eigenvalues of $\frakf$. 

\vspace{1.5mm}

Now we reduce to the case where $\Q_{\mathfrak{f}}=\mathbb{Q}$ to apply Theorem \ref{thm:ES}, after possibly enlarging $V$ by an effective amount. This step uses standard ideas originally due to Mazur that can be found in \cite[Section 4]{bs04} and  \cite[Proposition 15.4.2]{cohen07}.  Next, we want to show that there is some elliptic curve $E'/K$ of conductor $\mathcal{N}_p$ having the same $L$-function as $\mathfrak{f}$.  We know that $E$ has potentially multiplicative reduction at $\widetilde\frakP$ and $p\mid \# \overline{\rho}_{E,p}(I_{\widetilde\frakP})$ by Lemma \ref{potmultred}.  We can conclude that there is an elliptic curve $E'=E_\frakf$ of conductor $\calN_p$ satisfying $\overline{\rho}_{E,p} \sim \overline{\rho}_{E',p}$ if we implement Theorem \ref{thm:ES} after possibly enlarging $V$ to ensure that $p \nmid (\Norm_{K/\mathbb{Q}}(\widetilde{\mathfrak{P}})\pm 1)$.

\vspace{1.5mm}

    In the general case, $K$ is not necessarily totally real. For $p$ large enough, Lemma \ref{eichlershimura} implies that we have an elliptic curve $E'/K$ of conductor $\frakN'$ dividing $\frakN_E$ such that $\overline{\rho}_{E,p}\sim \overline{\rho}_{E',p}$, which completes (i) and (ii) as $\frakN_E$ is supported on $S_K$ by Lemma~\ref{conductoroffreycurve}. 
    
\vspace{1.5mm}

Now, suppose that $E'$ is $2$-isogenous to an elliptic curve $E''$. As the isogeny induces
an isomorphism $E'[p] \equiv E''[p]$ of Galois modules ($p\neq 2$), we get $\overline{\rho}_{E,p}\sim \overline{\rho}_{E',p}\sim\overline{\rho}_{E'',p}$ completing the proof of (ii). After possibly replacing $E'$ by $E''$, we can suppose that $E'$ has a $K$-rational point of order $2$ giving us (iii).

    \vspace{1.5mm}

    It remains to prove  $v_{\widetilde\frakP}(j')<0$ where $j'$ is the $j$-invariant of $E'$. By Lemma~\ref{potmultred}, the prime $p$ divides the size of $\overline{\rho}_{E,p}(I_{\widetilde\frakP})$. 
	It now follows from Lemma~\ref{imageofinertia} that $v_{\widetilde\frakP}(j')<0$ since the sizes of  
	$\overline{\rho}_{E,p}(I_{\widetilde\frakP})$ and $\overline{\rho}_{E',p}(I_{\widetilde\frakP})$ are equal.
\end{proof}

\section{Asymptotic Results}\label{asy}

Before going into the details of the proof, we will state the relationship between solutions of $S$-unit equations and certain families of elliptic curves.

\begin{lemma}\label{lmm:bijection}
    Let $K$ be a number field and $E/K$ be an elliptic curve with a $K$-rational point of order $2$. Then $E$ has a model of the form
    \begin{equation*}
        E: Y^2=X^3+aX^2+bX.
    \end{equation*}
    Further, there is a bijection between
    \begin{equation*}
        \{E/K \text{ with a } K \text{-torsion of order } 2 \text{ up to } \Kbar \text{-isomorphism} \} \longrightarrow \P^1(K)-\{4,\infty\}
    \end{equation*}
    via the map $E\mapsto \lambda:=\frac{a^2}{b}$.
\end{lemma}
\begin{proof}
This is \cite[Lemma 15 (i)]{moc22}.
\end{proof}

Let $S$ be a finite set of primes of $\O_K$ and let $\O_S$ be the ring of $S$-integers and $\mathcal{O}_{S}^\times$ be the group of $S$-units.

\begin{lemma}\label{lmm:bijection2}
    Let $K$ be a number field and $S$ be a set of finite primes of $K$. If $S$ contains the primes above $2$, we then have the following bijection
    \begin{equation*}
       \left \{\displaystyle{\substack{E/K \text{ with a } K \text{-torsion of order } 2 \text{ with potentially } \\
       \text{good reduction outside } S\text{ up to } \Kbar \text{-isomorphism}}} \right\} \longrightarrow \mathcal{O}_{S}^\times
    \end{equation*}
    via the map $E\mapsto \mu:=\lambda-4\in\mathcal{O}_{S}^\times $, where $\lambda$ is given as in Lemma \ref{lmm:bijection}.
\end{lemma}
\begin{proof}
This is \cite[Lemma 16 (i)]{moc22}.
\end{proof}

\begin{lemma}\label{lmm:fractionalideal}
 Let $K$ be a number field and $S$ be a set of finite primes of $K$. Let $E/K$ be an elliptic curve
with good reduction outside $S$. Suppose $S$ contains the primes above $2$ and $E$ has a $K$-torsion point of order 2. Let $(\lambda,\mu) \in \O_S \times \mathcal{O}_{S}^\times$ correspond to $E$ as in Lemma \ref{lmm:bijection2} and therefore satisfies $\lambda-\mu=4$. Then $(\lambda)\O_K = I^2J$ where $I,J$ are fractional ideals with $J$ being an $S$-ideal.
\end{lemma}
\begin{proof}
This is \cite[Lemma 17]{moc22}.
\end{proof}

\subsection{Proof of Theorem~\ref{maintheoremA}} 
 
The proof follows closely the proof of \cite[Theorem 3]{moc22}. Let $K$ be a number field with $\Cl_{S_K}(K)[2]=\{1\}$.  If $K$ is not totally real, we assume both Conjecture~\ref{conj1} and Conjecture~\ref{conj2} hold for $K$. We have so far proven that a putative non-trivial solution $(a,b,c)$ to \eqref{maineqn}  with $\widetilde{\frakP}\mid b$ yields an elliptic curve $E'/K$ with a $K$-rational point of order $2$ having potentially good reduction outside $S_K$ by Theorem~\ref{levellowerinandeichlershimura}. Hence, by Lemma \ref{lmm:bijection} we get a model
\begin{equation*}
    E': Y^2=X^3+a'X^2+b'X
\end{equation*}
with arithmetic invariants $\Delta_{E'}=2^4b'^2(a'^2-4b')$ and $j_{E'}=2^8\frac{(a'^2-3b')^3}{b'^2(a'^2-4b')}$. Further, by Theorem \ref{levellowerinandeichlershimura} (i), we know that $E'$ has good reduction outside $S_K$ which implies that $v_{\frakq}(j_{E'})\geq 0$ for
$\frakq\not\in S_K$, whence $j_{E'}\in \O_{S_K}$. Set $\lambda:=\frac{a'^2}{b'}$ and $\mu:=\lambda-4=\frac{a'^2-4b'}{b'}$. We now claim that $\lambda=u\gamma^2$ where $u$ is an $S_K$-unit. Indeed, it follows from Lemma \ref{lmm:fractionalideal} that
\begin{equation*}
    (\lambda)\O_K=I^2J,
\end{equation*}
where $J$ is an $S_K$-ideal. Therefore, $[I]^2=[J]$ as elements of the class group $\Cl(K)$ and $[J]\in \langle [\frakP]\rangle_{\frakP\in S_K}$. This implies that $[I]\in \Cl_{S_K}(K)[2]$. As we assume that $\Cl_{S_K}(K)[2]=\{1\}$, we see that $[I]\in \langle [\frakP]\rangle_{\frakP\in S_K}$. In particular, we have $I=\gamma \widetilde{I}$, where $\widetilde{I}$ is an $S_K$-ideal and $\gamma\in \O_K$. Thus,
\begin{equation*}
    (\lambda)\O_K=(\gamma)^2\widetilde{I}^2J,
\end{equation*}
where both $\widetilde{I}$ and $J$ are $S_K$-ideals. It then follows that $\left(\tfrac{\lambda}{\gamma^2}\right)\O_K$ is an $S_K$-ideal, which implies that $u:=\frac{\lambda}{\gamma^2}$ is an $S_K$-unit. By setting $\alpha:=\frac{\mu}{u}\in \mathcal{O}_{S_K}^\times$ and $\beta:=\frac{4}{u}\in \mathcal{O}_{S_K}^\times$, we obtain
\begin{equation*}
    \alpha+\beta=\gamma^2.
\end{equation*}

Since we assume $|v_{\widetilde{\frakP}}(\alpha/\beta)|\leq 6v_{\widetilde{\frakP}}(2)$, using $\alpha/\beta=\mu/4$ we see that
\begin{equation*}
    -4 v_{\widetilde{\frakP}}(2) \leq v_{\widetilde{\frakP}}(\mu) \leq 8 v_{\widetilde{\frakP}}(2).
\end{equation*}

Note that $j_{E'}=2^8(\mu+1)^3\mu^{-1}$, whence
\begin{equation*}
    v_{\widetilde{\frakP}}(j_{E'})=8v_{\widetilde{\frakP}}(2)+3v_{\widetilde{\frakP}}(\mu+1)-v_{\widetilde{\frakP}}(\mu).
\end{equation*}

We have three cases to consider depending on the valuation of $\widetilde{\frakP}$ at $\mu$:

\begin{itemize}
    \item [\textbf{Case (1):}] Suppose $v_{\widetilde{\frakP}}(\mu)= 0$, which implies that $v_{\widetilde{\frakP}}(\mu+1)\geq 0$. Therefore, we see that $v_{\widetilde{\frakP}}(j_{E'})\geq 0$, which gives a contradiction.
    
    \item [\textbf{Case (2):}] Suppose $v_{\widetilde{\frakP}}(\mu)>0$, which implies that $v_{\widetilde{\frakP}}(\mu+1)=0$. It follows from $v_{\widetilde{\frakP}}(\mu) \leq 8 v_{\widetilde{\frakP}}(2)$ that $v_{\widetilde{\frakP}}(j_{E'})\geq 0$, which gives a contradiction.
    
    \item [\textbf{Case (3):}]  Suppose $v_{\widetilde{\frakP}}(\mu)<0$, which implies that $v_{\widetilde{\frakP}}(\mu+1)=v_{\widetilde{\frakP}}(\mu)$. It follows from $  -4 v_{\widetilde{\frakP}}(2) \leq v_{\widetilde{\frakP}}(\mu)$ that $v_{\widetilde{\frakP}}(j_{E'})\geq 0$, which gives a contradiction.
\end{itemize}

All three cases lead to contradictions; hence, we conclude the proof of Theorem \ref{maintheoremA}.

\subsection{The proof of Theorem \ref{maintheoremB}}

The result will follow from Theorem \ref{maintheoremA}. Consider the equation
\begin{equation}\label{S_unit_equation}
    \alpha + \beta = \gamma^2, \quad \alpha, \beta \in \mathcal{O}_{S_K}^\times, \ \gamma \in \mathcal{O}_{S_K}.
\end{equation}

There is a natural scaling action of $\mathcal{O}_{S_K}^\times$ on the solutions. We regard two solutions 
\[
(\alpha_1, \beta_1, \gamma_1) \sim_2 (\alpha_2, \beta_2, \gamma_2)
\]
as equivalent if there exists some $\varepsilon \in \mathcal{O}_{S_K}^\times$ such that
\[
\alpha_2 = \varepsilon^2 \alpha_1, \quad \beta_2 = \varepsilon^2 \beta_1, \quad \text{and} \quad \gamma_2 = \varepsilon \gamma_1.
\]

Note that by \cite[Theorem 39]{moc22}, Equation \ref{S_unit_equation} has a finite number of solutions up to the equivalence relation $\sim_2$. Further, these are effectively computable.

Let $(\alpha,\beta,\gamma)\in \left(\calO_{S_K}^\times \right)^2\times \calO_{S_K}$ be a solution to the equation above. By scaling the equation by 
a proper element $\varepsilon \in \mathcal{O}_{S_K}^\times$ and swapping $\alpha$ and $\beta$ if necessary, we may assume $0\leq v_{\frakP}(\beta)\leq v_{\frakP}(\alpha)$ with $v_{\frakP}(\beta)\in \{0,1\}$ and $v_{\frakq}(\beta)=0$ for every odd prime $\frakq\in S_K$. One can now apply the same argument presented in the proof of \cite[Theorem 5]{moc22} to conclude the proof.

\subsection{The proof of Theorem \ref{maintheoremC}}

Let $q$ be a rational prime such that $q\geq 13$ and $q\equiv 5 \pmod{8}$, and $K=\mathbb Q(\sqrt{q})$. Then $2$ is inert in $K$, and it follows from that $2\nmid h_K^+$.  Let $\ell\geq 29$ be a prime satisfying $\ell \equiv 5 \pmod 8$ and $\left(\frac{q}{\ell} \right)=-1$. We then have that both $2$ and $\ell$ are inert in $K$, so $S_K=\{2,\ell\}$.

It follows from Theorem \ref{maintheoremB} that it is enough to check that every solution $(\alpha,\gamma)\in \calO_{S_K}^\times\times \calO_{S_K}$ with $v_{\frakP}(\alpha)\geq 0$ to the equation $\alpha+1=\gamma^2$ satisfies  $v_{\frakP}(\alpha)\leq 6$. From the equation, one can obtain that $(\gamma+1)(\gamma-1)=\alpha$. Set $x=\frac{(\gamma+1)}{2}$ and $y=\frac{(1-\gamma)}{2}$. Since they are the factors of the $S_K$-unit $\alpha$, we have $x,y\in \calO_{S_K}^\times$. It follows from \cite[Lemma 7.4]{HD} that the only solutions to the $S_K$-unit equation $x+y=1$ are $(-1,2), (1/2, 1/2), (2,-1)$. This leads to $ \alpha \in \{-1, 8\}$, and hence  $ v_{\frakP}(\alpha) \in \{0, 3\}$, proving $  v_{\frakP}(\alpha) \leq 6$. Thus, we can conclude the proof by Theorem \ref{maintheoremB}.

\section{Effective results over real and imaginary quadratic fields for signature $(p,p,2)$}\label{eff}

In this section, we study the equation
\begin{equation}\label{dfermat}
    x^p+dy^p=c^2
\end{equation} over various imaginary and real quadratic fields in which there is a unique prime $\frakP$ lying over $2$. This list currently contains $d=3,11,19,43$ for imaginary quadratic fields $\mathbb{Q}(\sqrt{-d})$ and $d=3,5,11,13,19,29$ for real quadratic fields $\mathbb{Q}(\sqrt{d})$; however, the list may grow if we can extend the form elimination computations. 

Recall that if $(a,b,c)$ is a solution to the Equation (\ref{dfermat}) with $a,b,c \in \mathcal{O}_K$, then the triple $(a,db,c)$ is called primitive if $a,db$ and $c$ are pairwise coprime. Our main results for the effective cases are as follows.

\begin{theorem}\label{maintheoremfofrimaginaryquadraticfields}
Let $K=\Q(\sqrt{-d})$ where $d\in \{3, 11,19, 43\}$. Assume that Conjecture~\ref{conj1} holds for $K$. Write $\frakP=2\mathcal O_K$.  Then, the equation 
     $$x^p+dy^p=z^2$$
     has no non-trivial solutions $(a,b,c)$ with $a,b,c \in \mathcal{O}_K$ such that $(a,db,c)$ is primitive and $\frakP\mid b$ when $p>C_K$ and $p \neq d$, where \[ C_K=
\left\{\begin{array}{l}
17 \text {, if $d=3,11,43$ } \\
19\text {, if $d=19$}.
\end{array}\right.\] 
 \end{theorem}
 
 For the real case, we do not need to assume Conjecture \ref{conj1} since we have Theorem \ref{level lowering}.
 
\begin{theorem}\label{thm1:dfermat}
    Let $K=\mathbb{Q}(\sqrt{d})$, where $d\in\{3,5,11,13,19,29\}$. Write $\frakP^e=2\mathcal{O}_K$. Then, the equation
      $$x^p+dy^p=z^2$$
    has no non-trivial solutions $(a,b,c)$ with $a,b,c \in \mathcal{O}_K$ such that $(a,db,c)$ is primitive and $\frakP\mid b$ when $p>C_K$ and $p \neq d$, where $$C_K=
\left\{\begin{array}{l}
17 \text {, if $d=3,5,11,13$ } \\
19\text {, if $d=19$ }\\
17\text {, if $d=29$. }
\end{array}\right.$$

\end{theorem}

\begin{remark}
 In the proofs of Theorem \ref{thm1:dfermat} and Theorem \ref{maintheoremfofrimaginaryquadraticfields}, we associate another elliptic curve, namely $E_1/K$ (given by Equation \ref{Frey-effective}) to the Diophantine equation (\ref{dfermat}) since the previously associated elliptic curve $E/K$ given by Equation \ref{Frey} is not useful to get effective results.
\end{remark}

\begin{remark}
  When $d=29$, we assumed $p\ne d$. This is due to the condition $(ii)$ of the level-lowering theorem.
\end{remark}

\begin{remark}
 We aimed to consider all real quadratic fields $\mathbb{Q}(\sqrt{d})$ of class number one with $3\le d \le 29$, in which there is a unique prime above $2$. This would include the cases $d = 7, 23$, in addition to those we discuss in Theorem \ref{thm1:dfermat}. However, we had to exclude $d = 7, 23$ since we were not able to eliminate Hilbert newforms in these cases.  Similarly, we wanted to consider all imaginary quadratic fields $\mathbb{Q}(\sqrt{-d})$ where $d\in \{1,3,11,19,43,67,163\}$.  However, we we were not able to eliminate Bianchi newforms for $d=1, 67, 163.$
\end{remark}

In order to study generalized Fermat equations with signatures $(p,p,2)$, we need to make sure that the Galois representation $\bar{\rho}_{E_1, p}$ attached to the Frey curve $E_1$ is irreducible. We start with introducing the necessary background to study Diophantine equations over number fields with an effective manner.

When the field $K$ is imaginary quadratic, the main strategy is to make sure that the representation $\bar{\rho}_{E_1, p}$ satisfies every condition of Serre's modularity conjecture, Conjecture \ref{conj1}.

When considering real quadratic fields, one of the main tools that will be used is a level-lowering theorem by Freitas and Siksek which was stated in Theorem \ref{level lowering}. Recall that all elliptic curves in consideration will be modular since $K$ is real quadratic by the work of Freitas, Le Hung and Siksek stated as Theorem \ref{modularityoffreycurveovertotallyreal}.

We first introduce an elliptic curve over the fields $K$ attached to a hypothetical solution $(a,b,c)$ of Equation \ref{dfermat} and determine its conductor.

Note that throughout this section, the assumptions of Theorems \ref{maintheoremfofrimaginaryquadraticfields} and \ref{thm1:dfermat} are assumed. In particular, $(a,db,c)$ is primitive.

\subsection{Conductor of $E_1$}

For a putative solution $(a,b,c)$ to (\ref{dfermat}), we associate the Frey elliptic curve,
	 \begin{equation} \label{Frey-effective}
	 E_1: Y^2+XY=X^3+\frac{c-1}{2^2}X^2+\frac{db^p}{2^6}X
	 \end{equation}
	 whose arithmetic invariants are given by $$\Delta_{E_1}=2^{-12}d^2(ab^2)^p, \; \displaystyle j_{E_1}=2^6\frac{(4a^p+db^p)^3}{d^2(ab^2)^p}$$ and $$c_4(E_1)=2^{-2}(4a^p+db^p),\; c_6(E_1)=2^{-3}c(db^p-8a^p).$$

\begin{lemma}\label{condE}
    Write $\frakP^e=2\mathcal{O}_K$ and $\mathfrak{D}^2=d\mathcal{O}_K$. We assume that $\frakP\mid b$ and $p>11$.  Then, the conductor of the curve $E_1$ is given by $$\mathcal{N}_{E_1}=\frakP^\epsilon \mathfrak{D}\prod_{\mathfrak{q}|ab, \mathfrak{q} \nmid 2d}\mathfrak{q},$$ where $\epsilon=0,1$.
    
 \end{lemma}

\begin{proof}
    Recall that the invariants $\Delta_{E_1},c_4(E_1)$, and $c_6(E_1)$ of the model $E_1$ are given by
    $$\Delta_{E_1}=2^{-12}d^2(ab^2)^p,\; c_4(E_1)=2^{-2}(4a^p+db^p),\; c_6(E_1)=2^{-3}c(db^p-8a^p).$$Suppose $\frakq \neq \frakP$ divides $\Delta_{E_1}$, which implies that  $dab$  is divisible by $\frakq$.  Since $a,db,c$ are pairwise coprime, $\frakq$ divides either $a$ or $db$. Therefore, $c_4(E_1)= 2^{-2}(4a^p+db^p)$ is not divisible by $\frakq$, i.e. $v_{\frakq}(c_4(E_1))=0$ and $v_{\frakq}(\calN_{E_1})=1$.  Hence, the given model is minimal and $E_1$ is semistable at $\frakq$.

    Now, assume that $\frakP\mid dab$. It is clear that $\frakP\nmid d$. Without loss of generality, assume that $\frakP\mid b$. Recall that we assumed $p>11$. Then, $v_\frakP(c_4(E_1))=v_\frakP(c_6(E_1))=0$ and $v_\frakP(\Delta_{E_1})\geq 0$ when  $p > 11$, hence $\epsilon=0,1$ via \cite[Tableau~IV]{pap} when $2$ is inert in the field $K$. Note that when $2$ ramifies in $K$, we again obtain  $\epsilon=0,1$ if $\frakP \mid b$ similarly, but using \cite[Tableau~V]{pap} this time.
\end{proof}

\begin{remark}\label{similaritywithLemma2.1}
    As in Lemma \ref{conductoroffreycurve}, one can show that the determinant of  the mod $p$ Galois representation $\overline{\rho}_{E_1,p}$ is the mod $p$ cyclotomic character $\chi_p$, and that $\overline{\rho}_{E_1,p}$ is finite flat at each prime $\frakp \mid p$ of $K$.
\end{remark}

\subsection{Irreducibility of $\bar{\rho}_{E_1, p}$ }
 
 In this section, we will prove the irreducibility of mod $p$ Galois representation $\bar{\rho}_{E_1, p}$ of the elliptic curve $E_1/K$ when $p>17$ where $K$ is one of the fixed real or imaginary quadratic fields given in Theorems \ref{maintheoremfofrimaginaryquadraticfields} and \ref{thm1:dfermat}. The following two lemmas will be useful for this purpose.  

\begin{lemma}\label{E1:potentiallymultiplicative}
    The elliptic curve $E_1/K$ has potentially multiplicative reduction at $\frakP$ when $\frakP\mid b$ and $p>11$.
\end{lemma}

\begin{proof}
  Observe that 
  \begin{eqnarray*}
      v_{\frakP}(c_4(E_1))&=&0\\
      v_{\frakP}(\Delta_{E_1})&=&-12v_{\frakP}(2)+2pv_{\frakP}(p).
  \end{eqnarray*}Therefore, $$v_{\frakP}(j_{E_1})=3v_{\frakP}(c_4(E_1))-v_{\frakP}(\Delta_{E_1})=12v_{\frakP}(2)-2pv_{\frakP}(b)<0$$when $p>11$.
\end{proof}

\begin{lemma}[\cite{FSANT}, Lemma $6.3$]\label{freitsiksek-exponent}
Let $E$ be an elliptic curve over a number field $K$ of conductor $\mathcal{N}$ and let $p \geq 5$ be a rational prime. Suppose $\bar{\rho}_{E, p}$ is reducible and write

$$
\bar{\rho}_{E, p} \sim\left(\begin{array}{cc}
\theta & * \\
0 & \theta^{\prime}
\end{array}\right),
$$where $\theta, \theta^{\prime}: G_K \rightarrow \mathbb{F}_p^*$ are characters. Write $\mathcal{N}_\theta$ and $\mathcal{N}_{\theta^{\prime}}$ for the respective conductors of these characters. Let $\mathfrak{q}$ be a prime of $K$ with $\mathfrak{q} \nmid p$.

\begin{enumerate}[(a)]
\item If $E$ has good or multiplicative reduction at $\mathfrak{q}$, then $v_{\mathfrak{q}}\left(\mathcal{N}_\theta\right)=v_{\mathfrak{q}}\left(\mathcal{N}_\theta^{\prime}\right)=0$.
\item If $E$ has additive reduction at $\mathfrak{q}$, then $v_{\mathfrak{q}}(\mathcal{N})$ is even and

$$
v_{\mathfrak{q}}\left(\mathcal{N}_\theta\right)=v_{\mathfrak{q}}\left(\mathcal{N}_{\theta^{\prime}}\right)=\frac{1}{2} v_{\mathfrak{q}}(\mathcal{N}).
$$
\end{enumerate}
\end{lemma}

\begin{proposition}\label{irred19} Consider the elliptic curve $E_1$ defined over $K=\Q(\sqrt{d})$ where $d \in \{-43,-19,-11,-3,3,5,11,13,19,29\}$. Assume $\frakP\mid b $ where $\frakP$ and $b$ are as above. Then 
$\bar{\rho}_{E_1, p}$ is irreducible for $p>17$ and $p \neq |d|$.
\end{proposition}

\begin{proof}
We will treat all the cases together. Assume not and say $\bar{\rho}_{E_1, p}$ is reducible. Then, we know that $$
\bar{\rho}_{E_1, p} \sim\left(\begin{array}{cc}
\theta & * \\
0 & \theta^{\prime}
\end{array}\right)
$$
for some characters $\theta$, $\theta^{\prime}$ : $G_K \rightarrow \mathbb{F}_p^{\times}$ such that $\theta \theta^{\prime}=\chi_p$, where $\chi_p$ is the $\bmod \;p$ cyclotomic character given by the action of $G_K$ on the group $\mu_p$ of $p$-th roots of unity.
By [\cite{Kraus96}, Lemme 1], we know that $\theta, \theta^{\prime}$ are unramified away from $p$ and the prime $\frakP$. Now, consider the following cases:

\begin{enumerate}
    \item Suppose that $p$ is coprime to  $\mathcal{N}_\theta$ or to $\mathcal{N}_{\theta^{\prime}}$. By replacing $E_1$ with the $p$-isogenous curve $E_1/ker\theta$, if needed, we may swap the characters $\theta$ and $\theta^{\prime}$ in the matrix representation of $\bar{\rho}_{E_1, p}$ and assume that $(p, \mathcal{N}_\theta)=1$. This allows us to assume that $\theta$ is unramified away from the prime $\frakP$.

     We shall use Lemma \ref{freitsiksek-exponent} to determine $\mathcal{N}_\theta$. By Lemma \ref{condE}, we know that the exponent of the prime $\frakP$ in the conductor of $E_1$ is either $0$ or $1$, i.e. $E_1$ has either good or multiplicative reduction at $\frakP$. Then Lemma \ref{freitsiksek-exponent} implies that $v_{\frakP}(\mathcal{N}_{\theta})=0$. 
     
     When $K$ is real quadratic, denote the two real places of $K$ by $\infty_1,\infty_2$. It follows that $\theta$ is a character of the ray class group for the modulus $\infty_1\infty_2$. However, when $K$ is imaginary quadratic $\theta$ is trivial. 
     Using \texttt{Magma}, we see that the ray class group is one of $\{1\}$, $\Z/2\Z$ in all cases. Therefore, the order of $\theta$ is $1$ or $2$. If the order of  $\theta$ is $1$, then $\theta$ is the trivial character. Therefore, $E_1$ has a point of order $p$ over $K$. But in [\cite{kamienny1992torsion}, Theorem $3.1$], Kamienny showed that the order of such a torsion point is at most $13$. If $\theta$ has order $2$, then $E_1$ has a $p$-torsion point defined over a quadratic extension $K$, say $L$. Then, $[L:\Q]=4$. According to Kamienny et al., \cite{der2021torsion}, the possible prime torsions of elliptic curves over number fields of degree $4$ imply that $p\le17$. In each case, we derive a contradiction as we assumed $p> 17$.

        \item Now, assume that $p$ is not coprime to $\mathcal{N}_\theta$ nor to $\mathcal{N}_{\theta^{\prime}}$. Observe that $p$ is not ramified in $K$ in each case, as we assumed $p\ne d$ when necessary. If $p$ is inert in $K$, then by [\cite{Kraus96}, Lemme 1], we see that $p$ divides precisely one of the conductors, which is a contradiction. So let us assume that $p$ splits in $K$, say $p\mathcal{O}_K=\mathfrak{p}_1 \mathfrak{p}_2$. Since $\mathfrak{p}_1, \mathfrak{p}_2$ are unramified, and $\bar{\rho}_{E_1, p}$ is reducible; by the contrapositive of [Corollary $6.2$, \cite{FSANT}] we see that $E_1$ cannot have good supersingular reduction at $\mathfrak{p}_1$ and $\mathfrak{p}_2$. But $E_1$ is semistable, hence $E_1$ has good ordinary or multiplicative reduction at these primes. It follows from [\cite{FSANT}, Proposition $6.1(ii)$] that
\begin{equation}\label{prop6.1}
    \left.\bar{\rho}_{E_1, p}\right|_{I_{\mathfrak{p_1}}} \sim\left(\begin{array}{ll}
\chi_p & * \\
0 & 1
\end{array}\right)\text{,}
\end{equation}which shows that exactly one of the isogeny characters is ramified at $\mathfrak{p}_1$, and the other one is ramified at $\mathfrak{p}_2$. We can assume that $\mathfrak{p}_1 \mid \mathcal{N}_\theta,\; \mathfrak{p}_1 \nmid \mathcal{N}_{\theta^{\prime}}$ and $\mathfrak{p}_2 \mid \mathcal{N}_{\theta^{\prime}}$, $\mathfrak{p}_2 \nmid \mathcal{N}_\theta$. By (\ref{prop6.1}), we see that $\left.\theta\right|_{I_{\mathfrak{p}_1}}=\left.\chi_p\right|_{I_{\mathfrak{p}_1}}$ and $\left.\theta^{\prime}\right|_{I_{\mathfrak{p}_2}}=\left.\chi_p\right|_{I_{\mathfrak{p}_2}}$.

Lemma \ref{E1:potentiallymultiplicative} shows that we are now dealing with the cases where the Frey curve $E_1/K$ has potentially multiplicative reduction at $\frakP$. Observe that $\theta^2$ is unramified everywhere except $\mathfrak{p}_1$. Indeed, this follows from the fact that since $E_1/K$ has potentially multiplicative reduction at $\frakP$, the restriction $\left.\bar{\rho}_{E_1, p}\right|_{I_{\frakP}}$ is up to semi-simplification equal to $\phi \oplus \phi\cdot\chi_p$, where $\phi$ is at worst a quadratic character. We also know that $\left.\theta^2\right|_{I_{\mathfrak{p}_1}}=\left.\chi_p^2\right|_{I_{\mathfrak{p}_1}}$. By [\cite{Turcas2018}, Lemma $4.3$] we have that
$$
\theta^2\left(\sigma_{\frakP}\right) \equiv N_{K_{\mathfrak{p}_1} / \mathbb{Q}_p}\left(\iota_{\mathfrak{p}_1}(\frakP)\right)^2 \pmod p,
$$where $\sigma_{\frakP}$ is the Frobenius element at $\frakP$, and $\iota_{\mathfrak{p}_1}$ is the inclusion map from $K$ into the completion of $K$ with respect to $\mathfrak{p}_1$. Also, by [\cite{SS}, Lemma $6.3$], we may assume that $\theta^2\left(\sigma_{\frakP}\right) \equiv 1\pmod p$. Let $b_K$ be a generator of the prime $\frakP$ in $K$. We have that
$$
N_{K_{\mathfrak{p}_1} / \mathbb{Q}_p}\left(\iota_{\mathfrak{p}_1}(\frakP)\right)^2-1=N_{K_{\mathfrak{p}_1} / \mathbb{Q}_p}\left(b_K\right)^2-1
=3,
$$which implies that $p\mid3$, a contradiction as $p> 5$ in each case.
\end{enumerate}
\end{proof}

When $K$ is a real quadratic number field, previous proposition gives us absolute irreducibility too. This is not the case when $K$ is imaginary, since irreducibility does not imply absolute irreducibility. However, in order to apply Conjecture \ref{conj1} we need absolute irreducibility of the Galois representation. We achieve this by the theorem below.

\begin{theorem}\label{surjectivityofGaloisrepresentation}
     Let $K = \Q(\sqrt{-d})$ with $d \in \{3, 11, 19, 43\}$.  Assume $\frakP\mid b $ where $\frakP$ and $b$ are as above. Then the Galois representation $\bar{\rho}_{E_1, p}$ is absolutely irreducible for $p>17$ and $p \neq d$.
\end{theorem}
\begin{proof}
     By Lemma \ref{condE}, the Frey curve $E_1$ is semistable when $\frakP \mid b$. It follows from Proposition \ref{irred19} that  $\overline{\rho}_{E_1,p}$ is irreducible for the prime $p>17$ and $p \neq |d|$. We have 
     \begin{equation*}
          v_{\frakP}(j_E)=6v_{\frakP}(2)+3v_{\frakP}(4a^p+db^p)-pv_{\frakP}(ab^2)<0
     \end{equation*}
     Therefore, by the theory of Tate curve, the inertia group $I_\frakP$ contains an element which acts on $E_1[p]$ via $\begin{pmatrix}
         1&1\\0&1
     \end{pmatrix}$ which has order $p$. By Theorem \ref{subgroups}, the image of $\overline{\rho}_{E_1,p}$ contains ${\rm SL}_2(\mathbb F_p)$ and is therefore absolutely irreducible.
\end{proof}

 \subsection{Abelianization}\label{section:abelianization}
In this section, we will deal with lifting of the mod $p$ eigenforms to complex ones which is needed for the proof of Theorem \ref{maintheoremfofrimaginaryquadraticfields}, meaning that this section is only for imaginary quadratic fields $K = \Q(\sqrt{-d})$ with $d \in \{3, 11, 19, 43\}$. 

The lifting of mod $p$ eigenforms to complex ones is ensured by Proposition \ref{eigenform}, which states that there is a constant $B(\frakN)$ such that, for all primes $p>B(\frakN)$, every weight-two $\mod p$ eigenform admits a lift to a complex eigenform. The key point of the proof is to make this bound explicit for the number fields appearing in Theorem \ref{maintheoremfofrimaginaryquadraticfields}. The method and the underlying theoretical framework for this lifting argument are described at the beginning of Section 5 of \cite{IKO}, to which we refer the reader to avoid unnecessary repetition.

Suppose that $(a,b,c)$ is a non-trivial solution to Equation \eqref{dfermat} with $a,b,c \in\mathcal{O}_K$ such that  $(a,db,c)$ is primitive and $\frakP\mid b$. Let $C_K$ denote the bound associated to the number field $K$ as given in Theorem~\ref{maintheoremfofrimaginaryquadraticfields}. It follows from Theorem \ref{surjectivityofGaloisrepresentation} that  $\overline{\rho}_{E_1,p}$ is absolutely irreducible for $p>C_K$, hence satisfies the hypotheses of Conjecture \ref{conj1}. Applying this conjecture,  we deduce that there exists a weight two, $\mod p$ eigenform $\theta$ over $K$ of level $\frakN_{E_1}$ such that for all primes $\frakq$ coprime to $p\frakN_{E_1}$, we have
\[
\tr(\overline{\rho}_{E_1,p}(\Frob_{\frakq}))=\theta(T_{\frakq}),
\]
where $T_{\frakq}$ denotes the Hecke operator at $\frakq$.

Recall that $\frakN_{E_1}$ denotes the Serre conductor of the residual representation $\overline{\rho}_{E_1,p}$, which is a power of $\frakP$ times $\frakD$. We now aim to lift this mod $p$ Bianchi modular form to a complex one. By Lemma \ref{condE}, the possible Serre conductors are as follows: $\frakN_{E_1} = \frakP^\epsilon\frakD$ where $\epsilon\in \{0,1\}$.

    We compute the abelianizations $\Gamma_0(\frakN_{E_1})^{\rm ab}$ by implementing the algorithm of {\c{S}}eng{\"u}n \cite{Sen11}. One can access to the relevant \texttt{Magma} codes online at \url{https://warwick.ac.uk/fac/sci/maths/people/staff/turcas/fermatprog}. We record here the primes $\ell$ that appear as orders of torsion elements of $\Gamma_0(\frakN_{E_1})^{\rm ab}$ for each number field:
        \begin{itemize}
        \item $K = \Q(\sqrt{-3})$:  For $\frakN_{E_1}=\frakQ$, we have $\ell=2$ and for $\frakN_{E_1}=\frakP\frakQ$ we have $\ell=3$.

          \item $K = \Q(\sqrt{-11})$: For $\frakN_{E_1}=\frakQ$, we have $\ell=2$ and for $\frakN_{E_1}=\frakP\frakQ$ we have $\ell=3$.

          \item $K = \Q(\sqrt{-19})$: For $\frakN_{E_1}=\frakQ$ or for $\frakN_{E_1}=\frakP\frakQ$ we have $\ell=3$.
          
        \item $K = \Q(\sqrt{-43})$: For $\frakN_{E_1}=\frakQ$, we have $\ell=3$ and for $\frakN_{E_1}=\frakP\frakQ$ we have $\ell=7$. 
    \end{itemize}

\subsection{Form elimination}\label{section:formelimination}

In this section, we will eliminate all possible forms attached to the elliptic curve i.e., all forms satisfying $\overline{\rho}_{E_1,p}\sim \overline{\rho}_{\frakf,\frakp}$. In order to achieve this, we need the following lemma. 

    \begin{lemma}\label{Bfcalculation}
        Let $\frakq$ be a prime ideal of $K$ not dividing $\frakP\frakQ$, and let $\frakf$ be a newform of level dividing $\frakP\frakQ$. Define the following
        \begin{equation*}
            \calA(\frakq)=\{a\in\Z : |a|\leq 2 \sqrt{\Norm(\frakq)}, \; \Norm(\frakq)+1-a \equiv 0 \pmod 2 \}.
        \end{equation*}
        If $\overline{\rho}_{E_1,p}\sim \overline{\rho}_{\frakf,\frakp}$, where $\frakp$ is the prime ideal of $\Q_\frakf$ lying above $p$, then $\frakp$ divides
        \begin{equation*}
            B_{\frakf,\frakq}:= \Norm(\frakq)\cdot \left((\Norm(\frakq)+1)^2- \frakf(T_\frakq)^2  \right)\cdot \prod_{a\in \calA(\frakq)}(a-\frakf(T_\frakq))\calO_{\Q_\frakf}.
        \end{equation*}
    \end{lemma}

    \begin{proof}
        The proof of this lemma is very similar to Lemma 7.1 in \cite{FSANT}.
    \end{proof}

We start with imaginary number fields. Let $K$ be one of the fields stated in Theorem \ref{maintheoremfofrimaginaryquadraticfields}. Assume that $C_K$ is the bound related to the number field $K$ given in Theorem \ref{maintheoremfofrimaginaryquadraticfields}, and that $p> C_K$.  It then follows that the $p$-torsion subgroups of $\Gamma_0(\frakN_{E_1})^{\rm ab}$ are all trivial, so the mod $p$ eigenforms must lift to complex ones. Hence, there exists a (complex) Bianchi modular form $\frakf$ over $K$ of level $\frakN_{E_1}$ such that for all prime ideals $\mathfrak{q}$ coprime to $p\frakN_{E_1}$ we have

$${\rm Tr}(\overline{\rho}_{E_1,p}({\rm Frob}_{\mathfrak{q}})) \equiv \mathfrak{f}(T_{\mathfrak{q}}) \pmod{\mathfrak{p}},$$
where $\frakp$ is a prime ideal of $\mathbb{Q}_\frakf$ lying above $p$ and $\mathbb{Q}_\frakf$ is the number field generated by the eigenvalues. Let us denote this relation by $\overline{\rho}_{E_1,p}\sim \overline{\rho}_{\frakf,\frakp}$. The final step of the proof is eliminating such forms and getting a contradiction that none of these forms can be associated to a putative solution of the equation.

    Using \texttt{Magma}, we computed the cuspidal newforms at the predicted levels, the fields $\Q_\frakf$, and eigenvalues $\frakf(T_\frakq)$ at the prime ideals $\frakq$ of norm less that $50$ for each imaginary quadratic number field $K = \Q(\sqrt{-d})$ with $d \in \{3, 11, 19, 43\}$. For each modular form $\frakf$ of level $\frakQ$ or of level $\frakP\frakQ$, we computed the ideal
    \begin{equation*}
        B_\frakf:=\sum_{\frakq\in S}B_{\frakf,\frakq},
    \end{equation*}
where $S$ denotes the set of prime ideals $\frakq\nmid \frakP\frakQ$ of $K$ of norm less than $50$. Set $C_\frakf:=\Norm_{\Q_\frakf/\Q}\left( B_\frakf \right)$. Using Lemma \ref{Bfcalculation} we can eliminate the forms as follows:

  \begin{itemize}
        \item $K = \Q(\sqrt{-3})$: There are no Bianchi modular forms of level $\frakQ$ or of level $\frakP\frakQ$. 
        \item $K = \Q(\sqrt{-11})$: There are no Bianchi modular forms of level $\frakP\frakQ$. There  is only one one Bianchi newform $\frakf$ at level $\frakQ$. It follows from Lemma \ref{Bfcalculation} that $C_\frakf$ is $45$.  

        \item $K = \Q(\sqrt{-19})$: There are two Bianchi newforms $\frakf_1,\frakf_2$ at level $\frakP\frakQ$, and there is one Bianchi newform $\frakf_3$ of level $\frakQ$. We compute $C_{f_i}$ for $i = 1, 2, 3$ using Lemma \ref{Bfcalculation} and find that the largest prime divisor of $C_{f_i}$'s is $7$ ($C_{f_i}\in\{525,2835\}$). 
        \item $K = \Q(\sqrt{-43})$: There are three Bianchi newforms $\frakf_1,\frakf_2$ at level $\frakP\frakQ$, and there are three Bianchi newforms $\frakf_3, \frakf_4, \frakf_5$ of level $\frakQ$. We compute $C_{f_i}$ for $i = 1, 2, 3,4,5$ using Lemma \ref{Bfcalculation} and find that the largest prime divisor of $C_{f_i}$'s is $11$ ($C_{f_i}\in\{5, 45, 49,55,297675\}$). 
    \end{itemize}

Now we continue with the real quadratic case. Let $K$ be one of the real quadratic number fields given in Theorem \ref{thm1:dfermat}. 

By Theorem \ref{level lowering}, we know that $\bar{\rho}_{E_1, p} \sim \bar{\rho}_{\mathfrak{f}, \omega}$ for some $\mathfrak{f}$ Hilbert newform $\mathfrak{f}$ of parallel weight $2$ at level $\mathcal{N}_p=\frakP^i\mathfrak{D}$ where $i=0,1$ and for some prime ideal $\omega$ of $\mathbb{Q}_f$ that lies above $p$. Now we will eliminate all such $\mathfrak{f}$. In the next section we check that all conditions of Theorem \ref{level lowering} are met, hence applying it to $E_1$ is possible.

\begin{itemize}
    \item When $K=\Q(\sqrt{3}),\Q(\sqrt{5})$, there are no Hilbert modular forms at levels $\mathfrak{D},\frakP\mathfrak{D}$.
    \item When $K=\Q(\sqrt{7})$, there are no Hilbert newforms at level $\mathfrak{D}$. There are two Hilbert newforms at level $\frakP\mathfrak{D}$. For both of them, the norm $C_\mathfrak{f}$ is divisible by $2$ and $3$.
    \item When $K=\Q(\sqrt{11})$, there are two Hilbert newforms at level $\mathfrak{D}$, for which $C_\mathfrak{f}$ is divisible by $3,5$ and $7$. There are two Hilbert newforms at level $\frakP\mathfrak{D}$. For both of them, the norm $C_\mathfrak{f}$ is divisible by $3,5$ and $7$.
    \item When $K=\Q(\sqrt{13})$, there is only one Hilbert modular form $\mathfrak{f}$ at level $\mathfrak{D}$ for which the norm $C_{\mathfrak{f}}$ is divisible by $3,5$ and $7$. There are two Hilbert newforms at level $\frakP\mathfrak{D}$ for which the norm $C_{\mathfrak{f}}$ is divisible by $3,5$ or $7$.
     \item When $K=\Q(\sqrt{19})$, there are four Hilbert newforms at level $\mathfrak{D}$ for which the norm $C_{\mathfrak{f}}$ is divisible by $3,5$ or $7$. There are $10$ Hilbert newforms at level $\frakP\mathfrak{D}$. For all of them, the norm $C_\mathfrak{f}$ is divisible by $3,5$ or $19$.
    \item  When $K=\Q(\sqrt{29})$, there are three Hilbert newforms at level $\mathfrak{D}$. For these modular forms $C_{\mathfrak{f}}$ is divisible by $5$ or $7$. There are five Hilbert newforms at level $\frakP\mathfrak{D}$, for which $C_\mathfrak{f}$ is divisible by $3,5$ or $7$.
\end{itemize}

\subsection{Proof of Theorem \ref{maintheoremfofrimaginaryquadraticfields} and Theorem \ref{thm1:dfermat}} In this section, we summarize the proofs of Theorem \ref{maintheoremfofrimaginaryquadraticfields} and Theorem \ref{thm1:dfermat}. Let $(a,b,c)$ be a non-trivial solution to Equation (\ref{dfermat}) such that $(a,db,c)$ is primitive and $\frakP\mid b$. To this solution, we attach the Frey curve $E_1/\mathbb{Q}(\sqrt{d})$ or $E_1/\mathbb{Q}(\sqrt{-d})$, having the model given in (\ref{Frey-effective}). 

For the proof of Theorem \ref{maintheoremfofrimaginaryquadraticfields}, we check that Serre's modularity conjecture, Conjecture \ref{conj1}, is applicable: Since the fields we work with in this case are all totally complex, the Galois representation $\bar{\rho}_{E_1, p}$ is odd. Further, we have $\det(\overline{\rho}_{E_1,p})=\chi_p$ and that $\overline{\rho}_{E_1,p}$ is finite flat at each prime $\frakp \mid p$ of $K$ (see Remark \ref{similaritywithLemma2.1}).  By Theorem \ref{surjectivityofGaloisrepresentation}, the representation is absolutely irreducible for $p>C_K$, where $C_K$ is the constant defined in Theorem \ref{maintheoremfofrimaginaryquadraticfields}. Hence, Conjecture \ref{conj1} predicts the existence of mod $p$ eigenforms at levels $\frakP^{\epsilon}\mathfrak{D}$, where $i=0,1$, all of which lift to complex eigenforms as we discussed in Section \ref{section:abelianization} since $p>C_K$. Finally, we eliminate each Bianchi modular form using Lemma \ref{Bfcalculation}.

For the proof of Theorem \ref{thm1:dfermat}, we first justify that the level-lowering theorem, Theorem \ref{level lowering}, is indeed applicable: Since the field $K$ is real quadratic, item $(i)$ of Theorem \ref{level lowering} is automatically satisfied. By Theorem \ref{modularityovertotallyreal}, we know that $E_1$ is modular, hence item $(ii)$ of Theorem \ref{level lowering} is satisfied. By Proposition \ref{irred19}, we know that the mod $p$ representation  $\bar{\rho}_{E_1, p}$ attached to the curve $E_1$ is irreducible for $p>17$, i.e. item $(iii)$ is also satisfied. Lemma \ref{condE} implies that item $(iv)$ of Theorem \ref{level lowering} holds. However, we still need to check if item $(v)$ of Theorem \ref{level lowering} is satisfied. The following lemma addresses this:

\begin{lemma}\label{pdivdel}
Let $p$ be a rational prime. Then, $p\mid v_{\mathfrak{p}}(\Delta_{E_1})$ for any prime ideal $\mathfrak{p}$ of $K$ such that $\mathfrak{p}\nmid 2d$.    
\end{lemma}

\begin{proof}
    Let $\mathfrak{p}\nmid 2d$ be a prime of $K$. Observe that 
    \begin{align*}
      v_{\mathfrak{p}}(\Delta_{E_1})&=-12v_{\mathfrak{p}}(2)+2v_{\mathfrak{p}}(d)+pv_{\mathfrak{p}}(ab^2)\\
      &=pv_{\mathfrak{p}}(ab^2).
    \end{align*}Therefore, $p\mid v_{\mathfrak{p}}(\Delta_{E_1})$.
\end{proof}

By the previous lemma, item $(v)$ of Theorem \ref{level lowering} is now satisfied. Using the conductor computation of $E_1$ given in Lemma \ref{condE}, Theorem \ref{level lowering} then implies the existence of Hilbert newforms $\mathfrak{f}$ at level $\mathcal{N}_p=\frakP^i\mathfrak{D}$ where $i=0,1$ satisfying $\bar{\rho}_{E_1, p} \sim \bar{\rho}_{\mathfrak{f}, \omega}$.

Finally, as we discussed in Section \ref{section:formelimination}, each such Hilbert newform $\mathfrak{f}$ can be eliminated when $p>C_K$, where the constant $C_K$ is defined as in Theorem \ref{thm1:dfermat}.

Therefore, Theorem \ref{maintheoremfofrimaginaryquadraticfields} and Theorem \ref{thm1:dfermat} hold.

\section*{Appendix: Effective results for $(p,p,3)$ over real quadratic fields}\label{app}

In \cite{isik2023ternary}, the authors have derived effective results for the equation $a^p+b^p=c^3$ over certain imaginary quadratic fields of class number one. In a recent work \cite{kara2025non}, the authors have studied the equation $Aa^p+Bb^p=Cc^3$ and proved asymptotic results over number fields. Moreover, they have studied the equation $a^p+db^p=c^3$ effectively over certain imaginary quadratic fields $\mathbb{Q}(\sqrt{-d})$ of class number one. 

In this section, we study the equation
\begin{equation}\label{dFermat-3}
    a^p+db^p=c^3
\end{equation}over the real quadratic fields $\Q(\sqrt{d})$ when $d=2,5,14$. Our main result is as follows:

\begin{theorem}\label{thm:(p,p,3)}
   Let $d=2,5,14$ and $K=\mathbb{Q}(\sqrt{d})$. Write $\lambda^e=3\mathcal{O}_K$. Then Equation \ref{dFermat-3} has no non-trivial solutions $(a,b,c)\in \mathcal{O}_K^3$ such that $(a,db,c)$ is primitive and $\lambda\mid b$ when $p>17$.
\end{theorem}

Let $(a,b,c)\in \mathcal{O}_K^3$ be a non-trivial solution to Equation (\ref{dFermat-3}) such that $(a,db,c)$ is primitive. To this triple, we attach the Frey elliptic curve:

\begin{equation}
    E=E_{a,b,c}:\text{ }Y^2+3cXY+db^pY=X^3
\end{equation}whose discriminant is $\Delta_E=3^3d^3(ab^3)^p$.

Lemma $2.4$ of \cite{kara2025non} immediately allows us to determine the conductor $\mathcal{N}_E$ of the elliptic curve $E$:

\begin{lemma}\label{lemma:cond-(p,p,3)}
    Write $\lambda^e=3\mathcal{O}_K$ and assume $\lambda\mid b$. The elliptic curve $E$ is semistable and has a $K$--rational point of order $3$. Moreover, the conductor $\mathcal{N}_E$ of the curve $E$ is given by: $$\mathcal{N}_E=\lambda^\epsilon\prod_{\mathfrak{q}|dab, \mathfrak{q} \nmid 3}\mathfrak{q},$$ where $\epsilon=\{0,1\}$ if $p>4$.
\end{lemma}

We will prove Theorem \ref{thm:(p,p,3)} via the Level-Lowering Theorem, namely Theorem \ref{level lowering}. The following lemma verifies condition $(ii)$ of the aforementioned theorem.

\begin{lemma}\label{pdivdel-(p,p,3)}
Let $p$ be a rational prime. Then, $p\mid v_{\mathfrak{p}}(\Delta_{E})$ for any prime ideal $\mathfrak{p}$ of $K$ such that $\mathfrak{p}\nmid 3d$.    
\end{lemma}

\begin{proof}
    Let $\mathfrak{p}\nmid 3d$ be a prime of $K$. Then, 
    \begin{align*}
      v_{\mathfrak{p}}(\Delta_{E})&=3v_{\mathfrak{p}}(3)+3v_{\mathfrak{p}}(d)+pv_{\mathfrak{p}}(ab^3)\\
      &=pv_{\mathfrak{p}}(ab^3).
    \end{align*}Therefore, $p\mid v_{\mathfrak{p}}(\Delta_{E})$.
\end{proof}

Recall that in order to apply Level-Lowering Theorem, one also needs the irreducibility of the mod $p$ Galois representation $\bar{\rho}_{E, p}$ attached to the Frey curve $E$:

\begin{proposition}\label{prop:irred-(p,p,3)}
  Assume $\lambda\mid b$ where $\lambda$ and $b$ are as above. Then, the representation $\bar{\rho}_{E, p}$ is irreducible for $p>17$. 
\end{proposition}

\begin{proof}
We will treat all the cases together. Assume not and say $\bar{\rho}_{E, p}$ is reducible. Then, we know that $$
\bar{\rho}_{E, p} \sim\left(\begin{array}{cc}
\theta & * \\
0 & \theta^{\prime}
\end{array}\right)
$$
for some characters $\theta$, $\theta^{\prime}$ : $G_K \rightarrow \mathbb{F}_p^{\times}$ such that $\theta \theta^{\prime}=\chi_p$, where $\chi_p$ is the mod $p$ cyclotomic character given by the action of $G_K$ on the group $\mu_p$ of $p$-th roots of unity.
By [\cite{Kraus96}, Lemme 1], we know that $\theta, \theta^{\prime}$ are unramified away from $p$ and the prime $\lambda$. Now, consider the following cases:

\begin{enumerate}
    \item Suppose that $p$ is coprime to  $\mathcal{N}_\theta$ or to $\mathcal{N}_{\theta^{\prime}}$. By replacing $E$ with the $p$-isogenous curve $E/ker\theta$, if needed, we may swap the characters $\theta$ and $\theta^{\prime}$ in the matrix representation of $\bar{\rho}_{E, p}$ and assume that $(p, \mathcal{N}_\theta)=1$. This allows us to assume that $\theta$ is unramified away from the prime $\lambda$.

     We shall use Lemma \ref{freitsiksek-exponent} to determine $\mathcal{N}_\theta$. By Lemma \ref{lemma:cond-(p,p,3)}, we know that the elliptic curve $E$ is semistable. In particular, the curve $E$ has either good or multiplicative reduction at $\lambda$. Then Lemma \ref{freitsiksek-exponent} implies that $v_{\lambda}(\mathcal{N}_{\theta})=0$. Denote the two real places of $K$ by $\infty_1,\infty_2$. It follows that $\theta$ is a character of the ray class group for the modulus $$\infty_1\infty_2.$$Using \texttt{Magma}, we see that the ray class group is one of $\{1\}$, $\Z/2\Z$ in all cases. Therefore, the order of $\theta$ is $1$ or $2$. If the order of  $\theta$ is $1$, then $\theta$ is the trivial character. Therefore, $E$ has a point of order $p$ over $K$. By Lemma \ref{lemma:cond-(p,p,3)}, $E$ also has $K$--rational point of order $3$. Now, $E(K)$ has a $3p$--torsion point but this is not possible due to the work of Kamienny et al. (\cite{kamienny1992torsion}, \cite{kenku1988torsion}) since we assumed $p\ge7$. If $\theta$ has order $2$, then $E$ has a $3p$-torsion point defined over a quadratic extension $K$, say $L$. Then, $[L:\Q]=4$. According to Kamienny et al., \cite{der2021torsion}, the possible prime torsions of elliptic curves over number fields of degree $4$ imply that $p\le17$. In each case, we derive a contradiction as we assumed $p> 17$.

    \item Now, assume that $p$ is not coprime to $\mathcal{N}_\theta$ nor to $\mathcal{N}_{\theta^{\prime}}$. Observe that $p$ is not ramified in $K$ in each case, so $p$ either is inert or splits in $K$. If $p$ is inert in $K$, then by [\cite{Kraus96}, Lemme 1], we see that $p$ divides precisely one of the conductors, which is a contradiction. So let us assume that $p$ splits in $K$, say $p\mathcal{O}_K=\mathfrak{p}_1 \mathfrak{p}_2$. Since $\mathfrak{p}_1, \mathfrak{p}_2$ are unramified, and $\bar{\rho}_{E, p}$ is reducible; by the contrapositive of [Corollary $6.2$, \cite{FSANT}] we see that $E$ cannot have good supersingular reduction at $\mathfrak{p}_1$ and $\mathfrak{p}_2$. But $E$ is semistable, hence $E$ has good ordinary or multiplicative reduction at these primes. It follows from [\cite{FSANT}, Proposition $6.1(ii)$] that
\begin{equation}\label{X}
    \left.\bar{\rho}_{E, p}\right|_{I_{\mathfrak{p_1}}} \sim\left(\begin{array}{ll}
\chi_p & * \\
0 & 1
\end{array}\right)\text{,}
\end{equation}which shows that exactly one of the isogeny characters is ramified at $\mathfrak{p}_1$, and the other one is ramified at $\mathfrak{p}_2$. We can assume that $\mathfrak{p}_1 \mid \mathcal{N}_\theta, \mathfrak{p}_1 \nmid \mathcal{N}_{\theta^{\prime}}$ and $\mathfrak{p}_2 \mid \mathcal{N}_{\theta^{\prime}}$, $\mathfrak{p}_2 \nmid \mathcal{N}_\theta$. By (\ref{X}), we see that $\left.\theta\right|_{I_{\mathfrak{p}_1}}=\left.\chi_p\right|_{I_{\mathfrak{p}_1}}$ and $\left.\theta^{\prime}\right|_{I_{\mathfrak{p}_2}}=\left.\chi_p\right|_{I_{\mathfrak{p}_2}}$. Hence $\theta$ is unramified away from $\mathfrak{p}_1$ and $\lambda$ since all bad places of E except possibly $\lambda$ are of potentially multiplicative reduction. By Lemma \ref{lemma:cond-(p,p,3)}, the curve $E$ has multiplicative or good
reduction at $\lambda$; therefore, we can say that $\theta$ is unramified away from $\mathfrak{p}_1$. The character $\left.\theta^2\right|_{I_{\mathfrak{p}_1}}=\left.\chi_p^2\right|_{I_{\mathfrak{p}_1}}$ is also unramified away from $\mathfrak{p}_1$ therefore by [\cite{Turcas2018}, Lemma 4.3], $\theta\left(\sigma_\lambda\right) \equiv \operatorname{Norm}_{K_{\mathcal{P}} / \mathbb{Q}_p}(\alpha)^2(\bmod p)$ where $\sigma_\lambda$ is the Frobenius automorphism at $\lambda=\langle 3\rangle$. We also know that by [\cite{SS}, Lemma 6.3] $\theta^2\left(\sigma_\lambda\right) \equiv 1(\bmod p)$ (note that $E$ has multiplicative reduction at $\lambda$.) Therefore, we have $p \mid \operatorname{Norm}_{K_{\mathcal{P}} / \mathbb{Q}_p}(\lambda)^2-1$, contradiction since $p>17$.
\end{enumerate}
\end{proof}

Now, by the Level-Lowering Theorem, Theorem \ref{level lowering}, we know that $\bar{\rho}_{E_1, p} \sim \bar{\rho}_{\mathfrak{f}, \omega}$ for some $\mathfrak{f}$ Hilbert newform $\mathfrak{f}$ of parallel weight $2$ at level $\mathcal{N}_p=\lambda^i\prod_{\mathfrak{D}\mid d }\mathfrak{D}$ where $i=0,1$ and for some prime ideal $\omega$ of $\mathbb{Q}_f$ that lies above $p$. Here, we similarly eliminate the arising Hilbert newforms by applying Lemma $46$ of \cite{KhS} with $t=3$ and using the primes $\mathfrak{q}\nmid \lambda\prod_{\mathfrak{D}\mid d }\mathfrak{D}$ of norm at most $50$.
\begin{itemize}
    \item When $d=2$, there are no Hilbert newforms at level $\lambda^i\mathfrak{D}$ when $i=0,1$.
    \item When $d=5$, there are no Hilbert newforms at level $\mathfrak{D}$. There is one Hilbert newform $\mathfrak{f}$ at level $\lambda\mathfrak{D}$, for which the prime divisors of the norm $C_f$ are $2,5,11$. Since we assumed $p>17$, equation $(\ref{dFermat-3})$ has no non-trivial solutions over $\Q(\sqrt{5})$ such that $(a,db,c)$ is primitive and $\lambda\mid b$.
    \item When $d=14$, let $\mathfrak{D}_1,\mathfrak{D}_2$ be the prime above $2,7$, respectively. Then, the possible levels for this case are $\lambda^ i\mathfrak{D}_1\mathfrak{D}_2$ for $i=0,1$. There are four Hilbert newforms at level $\mathfrak{D}_1\mathfrak{D}_2$. For three of these forms, $C_f$ is divisible by $2,3,5$ or $11$. The remaining form $\mathfrak{f}$ is rational and satisfies $C_{\mathfrak{f}}=0$. Using the LMFDB, we observe that it should correspond to the form with LMFDB label either $2.2.56-14.1-a$ or $2.2.56-14.1-b$. One of the elliptic curves in the isogeny classes given in the LMFDB label $2.2.56-14.1-a,b$ corresponds to the forms $\mathfrak{f}$. All the elliptic curves in the previous isogeny classes have potentially good reduction at $\lambda$. Since the elliptic curve $E$ has potentially multiplicative reduction at the prime $\lambda$, we have a contradiction due to the inertia argument. There are ten Hilbert newforms at level $\lambda\mathfrak{D}_1,\mathfrak{D}_2$, for which $C_{\mathfrak{f}}$ is divisible by $2,3,5$ or $7$. Since we assumed $p>17$, we deduce Theorem \ref{thm:(p,p,3)}.

\end{itemize}

\bibliographystyle{alpha} 
\bibliography{ref}
\end{document}